\documentclass{amsart}
\usepackage{graphicx, color}
\usepackage{amscd}
\usepackage{amsmath,empheq}
\usepackage{amsfonts}
\usepackage{amssymb}
\usepackage{mathrsfs}

\usepackage[all]{xy}

\newtheorem{theorem}{Theorem}

\newtheorem{prop}[theorem]{Proposition}
\newtheorem{remark}{Remark}

\newenvironment{proof-sketch}{\noindent{\bf Sketch of Proof}\hspace*{1em}}{\qed\bigskip}

\everymath{\displaystyle}

\newcommand{\RR}{\mathbb R}
\newcommand{\NN}{\mathbb N}

\renewcommand{\leq}{\leqslant}

\renewcommand{\geq}{\geqslant}

\begin{document}

\title[Periodic solutions for a class of evolution inclusions]{Periodic solutions for a class of evolution inclusions}

\author[N.S. Papageorgiou]{Nikolaos S. Papageorgiou}
\address[N.S. Papageorgiou]{National Technical University, Department of Mathematics,
				Zografou Campus, 15780 Athens, Greece}
\email{\tt npapg@math.ntua.gr}

\author[V.D. R\u{a}dulescu]{Vicen\c{t}iu D. R\u{a}dulescu}
\address[V.D. R\u{a}dulescu]{Faculty of Applied Mathematics, AGH University of Science and Technology, al. Mickiewicza 30, 30-059 Krak\'ow, Poland \& Institute of Mathematics ``Simion Stoilow" of the Romanian Academy, P.O. Box 1-764, 014700 Bucharest, Romania}
\email{\tt vicentiu.radulescu@imar.ro}

\author[D.D. Repov\v{s}]{Du\v{s}an D. Repov\v{s}}
\address[D.D. Repov\v{s}]{Faculty of Education and Faculty of Mathematics and Physics, University of Ljubljana, 1000 Ljubljana, Slovenia}
\email{\tt dusan.repovs@guest.arnes.si}


\begin{abstract}
We consider a periodic evolution inclusion defined on an evolution triple of spaces. The inclusion involves also a subdifferential term. We prove existence theorems for both the convex and the nonconvex problem, and we also produce extremal trajectories. Moreover, we show that every solution of the convex problem can be approximated uniformly by certain extremal trajectories (strong relaxation). We illustrate our results by examining a nonlinear parabolic control system.
\end{abstract}

\maketitle

\section{Introduction}

Let $T=[0,b]$ and let $(X,H,X^*)$ be an evolution of spaces (see Section 2). We assume that $X$ is embedded compactly into $H$. In this paper, we study the following periodic evolution inclusion
\begin{equation}\label{eq1}
	\left\{
		\begin{array}{l}
			-u'(t)\in A(t,u(t)) + \partial\varphi(u(t)) + F(t,u(t))\ \mbox{for almost all}\ t\in T,\\
			u(0) = u(b).
		\end{array}
	\right\}
\end{equation}

In this problem, $A:T\times X\rightarrow X^*$ is a map which is a measurable in $t\in T$ and monotone in $x\in X$. Also, $\varphi\in \Gamma_0(H)$ (see Section 2) and $\partial\varphi(\cdot)$ denotes the subdifferential of $\varphi$ in the sense of convex analysis. Finally, $F:T\times H\rightarrow2^{H}\backslash\{\emptyset\}$ is a multivalued perturbation.

Periodic problems for evolution inclusions have been studied either with $\varphi\equiv0$ (see Hu \& Papageorgiou \cite[Section 1.5]{5}, Xue \& Zheng \cite{11}) or with $A\equiv0$ (see Papageorgiou \& R\u adulescu \cite{9} and Papageorgiou, R\u adulescu \& Repov\v{s} \cite{9bis}). In (\ref{eq1}) both terms are present and this distinguishes the present work from  the aforementioned papers. Their methods and techniques are not applicable here. We prove existence theorems for the ``convex" problem (that is, $F$ has convex values) and for the ``nonconvex" problem (that is, $F$ has nonconvex values). We also prove the existence of extremal trajectories, that is, we produce solutions which move through the extreme points of the multivalued perturbation $F(t,x)$. Moreover, we show that every solution of the convex problem can be approximated in the $C(T,H)$-norm by certain extremal trajectories (strong relaxation). In the final part of this paper we illustrate our results by examining a parabolic distributed parameter system.

\section{Mathematical background}

The tools that we use in the study of problem (\ref{eq1}) come from multivalued analysis and from the theory of operations of monotone type. A detailed presentation of these theories can be found in the books of Hu \& Papageorgiou \cite{4} and Zeidler \cite{12}.

Let $(\Omega,\Sigma)$ be a measurable space and $V$ a separable Banach space. Throughout this work we will use the following notations:
$$\begin{array}{ll}
	P_{f_{(c)}}(V) = \left\{C\subseteq V:\ \mbox{$C$ is nonempty, closed (and convex)}\right\},\\
	P_{(w)k(c)}(V) = \left\{C\subseteq V:\ \mbox{$C$ is nonempty, ($w$)-compact (and convex)}\right\}.
\end{array}.$$

A multifunction (set-valued function) $F:\Omega\rightarrow2^{V}\backslash\{\emptyset\}$ is a said to be ``graph measurable", if
$$
{\rm Gr}\,F = \left\{(\omega, v)\in\Omega\times V: v\in F(\omega)\right\}\in\Sigma\otimes B(V),
$$
where $B(V)$ is the Borel $\sigma$-field of $V$. A multifunction $G:\Omega\rightarrow P_f(V)$ is ``measurable", if for all $v\in V$, the function
\begin{equation*}
	\omega\mapsto d(v,F(\omega))\equiv \inf\left\{||v-y||_V:y\in F(\omega)\right\}
\end{equation*}
is $\Sigma$-measurable. For multifunctions with values in $P_f(V)$, measurability implies graph measurability, while the converse is true if there is a $\sigma$-finite measure $\mu$ on $\Sigma$ and $\Sigma$ is $\mu$-complete.

Suppose that $(\Omega,\Sigma,\mu)$ is a $\sigma$-finite measure space and $F:\Omega\rightarrow2^V\backslash\{\emptyset\}$. For $1\leq p\leq\infty$, we define
\begin{equation*}
	S^p_F = \left\{h\in L^p(\Omega,V):h(\omega)\in F(\omega)\ \mu-\mbox{almost everywhere}\right\}.
\end{equation*}

A straightforward application of the Yankov-von Neumann-Aumann selection theorem (see Theorem 2.14 in Hu \& Papageorgiou \cite[p. 158]{4}), implies that
\begin{equation*}
	``S^p_F\neq\emptyset\ \mbox{if and only if}\ \inf\{||y||_V:y\in F(\omega)\}\in L^p(\Omega)."
\end{equation*}

The set $S^p_F$ is ``decomposable" in the sense that, if $(C,h_1,h_2)\in\Sigma\times S^p_F\times S^p_F$ then $\chi_C h_1 + \chi_{\Omega\backslash C}h_2\in S^p_F$. Since $\chi_{\Omega\backslash C}=1-\chi_C$, decomposability formally looks like the notion of convexity, only now the coefficients in the linear combination are functions. In fact, decomposable sets exhibit some properties which are similar to those of convex sets (see Hu \& Papageorgiou \cite[Section 2.3]{4}).

Suppose now that $Z$ and $Y$ are Hausdorff topological spaces and $F:Z\rightarrow2^Y\backslash\{\emptyset\}$. We say that $F(\cdot)$ is ``upper semicontinuous (usc)" (resp. ``lower semicontinuous (lsc)"), if for all open $U\subseteq Y$ the set $F^+(U)=\{z\in Z: F(z)\subseteq U\}$ (resp. $F^-(U)=\{z\in Z: F(z)\cap U\neq \emptyset\}$) is open. If $F(\cdot)$ has closed values and is usc, then ${\rm Gr}\,F\subseteq Z\times Y$ is closed. The converse is true if $F(\cdot)$ is locally compact (that is, for every $z\in Z$, we can find a neighbourhood $\mathcal{U}$ of $z$ such that $\overline{F(\mathcal{U})}\subseteq Y$ is compact). Also, if $Y$ is a metric space, then $F:Z\rightarrow 2^Y\backslash\{\emptyset\}$ is lsc if and only if for all $y\in Y$, the mapping $z\mapsto d(y,F(z)) = \inf\{d(y,v):v\in F(z)\}$ is an upper semicontinuous $\RR_+$-valued function.

Suppose that $Y$ is a metric space. On $P_f(Y)$ we can define a generalized metric, known as the ``Hausdorff metric", by
\begin{equation*}
	h(C,E)=\sup\left\{|d(u,C) - d(u,E)|: u\in Y\right\} = \max\{\sup_{c\in C}d(C,E), \sup_{e\in E}d(e,C)\}\ \mbox{for all}\ C,E\subseteq Y.
\end{equation*}

If $Y$ is a complete metric space, then so is $(P_f(Y),h)$. A multifunction $F:Z\rightarrow P_f(Y)$ is said to be ``h-continuous", if it is continuous from $Z$ into $(P_f(Y),h)$.

Suppose that $V,Y$ are Banach spaces and assume that $V$ is embedded continuously and densely into $Y$ (denoted by $V\hookrightarrow Y$). Then
\begin{itemize}
	\item [(a)] $Y^*$ is embedded continuously in $V^*$;
	\item [(b)] if $V$ is reflexive, then $Y^*\hookrightarrow V^*$.
\end{itemize}

A triple of spaces $(X,H,X^*)$ is said to be an ``evolution triple", if the following properties hold:
\begin{itemize}
	\item [(a)] $X$ is a separable, reflexive Banach space;
	\item [(b)] $H$ is a separable Hilbert space which we identify with its dual (that is, $H^*=H$);
	\item [(c)] $X\hookrightarrow H$ (hence $H\hookrightarrow X^*$).
\end{itemize}

By $||\cdot||$ (resp. $|\cdot|$, $||\cdot||_*$) we denote the norm of $X$ (resp. of $H,X^*$). Property (c) above implies that
\begin{equation*}
	|\cdot|\leq\hat{c_1}||\cdot||\ \mbox{and}\ ||\cdot||_*\leq\hat{c_2}|\cdot|\ \mbox{for some}\ \hat{c_1},\hat{c_2}>0.
\end{equation*}

We denote by $\langle\cdot,\cdot\rangle$ the duality brackets for the pair $(X^*,X)$ and by $(\cdot,\cdot)$ the inner product of $H$. We have
\begin{equation*}
	\langle\cdot,\cdot\rangle|_{H\times X} = (\cdot,\cdot).
\end{equation*}

Let $T=[0,b]$ and $1<p<\infty$. By $p'\in(1,\infty)$ we denote the conjugate exponent of $p$, that is, $\frac{1}{p}+\frac{1}{p'}=1$. We define
\begin{equation*}
	W_p(T) = \left\{u\in L^p(T,X):u'\in L^{p'}(T,X^*)\right\}.
\end{equation*}

Here, the derivative $u'$ is understood in the sense of vector-valued distributions. If $u\in W_p(T)$, then if we view $u(\cdot)$ as an $X^*$-valued function, then $u(\cdot)$ is absolutely continuous, hence it is differentiable almost everywhere. This derivative coincides with the distributional one and we have
\begin{equation*}
	W_p(T)\subseteq AC^{1,p'}(T,X^*) = W^{1,p'}((0,b),X^*).
\end{equation*}

We endow $W_p(T)$ with the norm
\begin{equation*}
	||u||_{W_p} = ||u||_{L^p(T,X)} + ||u'||_{L^{p'}(T,X^*)}\ \mbox{for all}\ u\in W_p(T).
\end{equation*}

Then $W_p(T)$ becomes a separable reflexive Banach space and we have
\begin{equation*}
	W_p(T)\hookrightarrow C(T,H)\ \mbox{and}\ W_p(T)\hookrightarrow L^p(T,H)\ \mbox{compactly}.
\end{equation*}

The elements of $W_p(T)$ satisfy the so-called ``integration by parts formula".
\begin{prop}\label{prop1}
	If $u,v\in W_p(T)$ and $\theta(t)=(u(t),v(t))$ for all $t\in T$, then $\theta(\cdot)$ is absolutely continuous and
	\begin{equation*}
		\frac{d\theta}{dt}(t) = \langle u'(t),v(t)\rangle + \langle u(t),v'(t)\rangle\ \mbox{for almost all}\ t\in T.
	\end{equation*}
\end{prop}

Let $V$ be a reflexive Banach space, $L:D\subseteq V\rightarrow V^*$  a linear maximal monotone map and $A:V\rightarrow2^{V^*}$. We say that $A$ is ``$L$-pseudomonotone", if the following conditions hold
\begin{itemize}
	\item [(\underline{a})] For every $v\in V, A(v)\in P_{wkc}(V^*)$.
	\item [(\underline{b})] $A$ is usc from every finite dimensional subspace of $V$ into $V^*$ furnished with the weak topology.
	\item [(\underline{c})] If $\{v_n\}_{n\geq1}\subseteq D, v_n\xrightarrow{w}v\ \mbox{in}\ V, L(v_n)\xrightarrow{w}L(v)\ \mbox{in}\ V^*,v^*_n\in A(v_n)$ \\
		$v^*_n\xrightarrow{w}v^*\ \mbox{in}\ V^*\ \mbox{and}\ \limsup_{n\rightarrow\infty}\langle v^*_n, v_n-v\rangle_V\leq0$, then $v^*\in A(v)$ and $\langle v^*_n,v_n\rangle_V\rightarrow\langle v^*,v\rangle_V$ (here by $\langle \cdot, \cdot\rangle_V$ we denote the duality brackets for the pair $(V^*,V)$).
\end{itemize}

Also, we say that $A(\cdot)$ is ``strongly coercive", if
\begin{equation*}
	\frac{\inf[\langle v^*,v\rangle_V:v^*\in A(v)]}{||v||_V}\rightarrow+\infty\ \mbox{as}\ ||v||_V\rightarrow+\infty.
\end{equation*}

$L$-pseudomonotone and strongly coercive maps exhibit remarkable surjectivity properties. More precisely, we have the following result (see Lions \cite{6} for $A(\cdot)$ single-valued) and Papageorgiou, Papalini \& Renzacci \cite{8} (for $A(\cdot)$ multivalued).
\begin{prop}\label{prop2}
	If $V$ is a reflexive Banach space, $L:D\subseteq V\rightarrow V^*$ is linear maximal monotone and $A:V\rightarrow2^{V^*}$ is bounded, $L$-pseudomonotone and strongly coercive, then $L+A$ is surjective (that is, $R(L+A)=V^*$).
\end{prop}

Suppose that $Y$ is a Banach space and $\{C_n\}_{n\geq1}\subseteq2^Y\backslash\{\emptyset\}$. We define
$$
\begin{array}{ll}
	w-\limsup_{n\rightarrow\infty}C_n=\left\{y\in Y: y=w-\lim_{k\rightarrow\infty}y_{n_k}, y_{n_k}\in C_{n_k}, n_1<n_2<\dots<n_k<\dots\right\}, \\
	\liminf_{n\rightarrow\infty}C_n = \left\{y\in Y: y=\lim_{n\rightarrow\infty}y_n, y_n\in C_n, n\in\NN\right\} = \left\{y\in Y: \lim_{n\rightarrow\infty}d(y,C_n)=0\right\}.
\end{array}
$$

We denote by $\Gamma_0(Y)$  the cone of all lower semicontinuous, convex proper functions. So, $\varphi\in\Gamma_0(Y)$ if $\varphi: Y\rightarrow\overline{\RR}=\RR\cup\{+\infty\}$ is lower semicontinuous, convex and dom $\varphi=\{y\in Y: \varphi(y)<+\infty\}$ (the effective domain of $\varphi$) is nonempty. By $\partial\varphi(\cdot)$ we denote the subdifferential in the sense of convex analysis. So,
\begin{equation*}
	\partial\varphi(y)=\left\{y^*\in Y^*: \langle y^*,h\rangle\leq\varphi(y+h) - \varphi(y)\ \mbox{for all}\ h\in Y\right\}.
\end{equation*}

It is well known that $\partial\varphi: Y\rightarrow2^{Y^*}$ is maximal monotone.

Given a nonempty set $C\subseteq Y$, we set
\begin{equation*}
	|C|=\sup\left\{||y||_Y:y\in C\right\}.
\end{equation*}

Finally, we denote by $||\cdot||_w$ the ``weak norm" on the Lebesgue-Bochner space $L^1(T,Y)$, defined by
\begin{equation*}
	||h||_w=\sup\left\{||\int^t_s h(\tau) d\tau||_Y: 0\leq s\leq t\leq b\right\}
\end{equation*}
or, equivalently,
\begin{equation*}
	||h||_w=\sup\left\{||\int^t_0 h(\tau) d\tau||_Y: 0\leq t\leq b\right\}.
\end{equation*}

The norm is equivalent to the Petils norm on $L^1(T,Y)$ (see Egghe \cite{2}). By $L^{1}_w(T,Y)$ we denote the space $L^1(T,Y)$ furnished with the weak norm.

\section{The ``convex" problem}

In this section we prove an existence theorem for the ``convex" problem, that is, we assume that the multivalued perturbation $F(t,x)$ is convex-valued.

We work on an evolution triple $(X,H,X^*)$ with $X\hookrightarrow H$ compactly. Hence $H\hookrightarrow X^*$ compactly, too. We impose two sets of hypotheses on the data $A(t,x)$ and $\partial\varphi(x)$.

\smallskip
$H(A)$: $A:T\times X\rightarrow X^*$ is a map such that
\begin{itemize}
	\item [(i)] for all $x\in X,\ t\mapsto A(t,x)$ is measurable;
	\item [(ii)] for almost all $t\in T,\ x\mapsto A(t,x)$ is demicontinuous (that is, $x_n\rightarrow x$ in $X$ implies $A(t,x_n)\xrightarrow{w}A(t,x)$) and
		\begin{equation*}
			c_0 ||x||^p\leq\langle A(t,x),x\rangle\ \mbox{for almost all}\ t\in T,\ \mbox{and all}\ x\in\RR;
		\end{equation*}
	\item [(iii)] $||A(t,x)||_*\leq a_1(t) + c_1||x||^{p-1}$ for almost all $t\in T$, and all $x\in X$, with $a_1\in L^{p'}(T)$ and $c_1>0$.
\end{itemize}

\smallskip
$H(\varphi)$: $\varphi\in\Gamma_0(H)$ is bounded above on bounded sets, for all $u\in L^p(T,X)$ we have $S^{p'}_{\partial\varphi(u(\cdot))}\neq\emptyset$, $\partial\varphi(0)\subseteq H$ is bounded and for all $(u,h), (u',h')\in {\rm Gr}\,\partial\varphi$ we have
\begin{equation*}
	c_0|u-u'|^2\leq(h-h',u-u').
\end{equation*}

Alternatively, we can assume the following conditions on $A$ and $\varphi$.

\smallskip
$H(A)'$: $A:T\times X\rightarrow X^*$ is a map such that hypotheses $H(A)'(i),(iii)$ are the same as hypotheses $H(A)(i),(iii)$ and
\begin{itemize}
	\item [(ii)] for almost all $t\in T,\ x\mapsto A(t,x)$ is demicontinuous,
	\begin{equation*}
		\begin{array}{ll}
			c_0||x||^p \leq \langle A(t,x),x\rangle\ \mbox{for almost all}\ t\in T,\ \mbox{and all}\ x\in X\ \mbox{with}\ c_0>0, \\
			c_0||x-y||^2\leq\langle A(t,x) - A(t,y), x-y\rangle\ \mbox{for almost all}\ t\in T,\ \mbox{and all}\ x,y\in X.
		\end{array}
	\end{equation*}
\end{itemize}

\smallskip
$H(\varphi)'$: $\varphi\in\Gamma_0(H)$ is bounded above on bounded sets, for all $u\in L^p(T,X)$, $S^{p'}_{\partial\varphi(u(\cdot))}\neq\emptyset$ and $\partial\varphi(0)\subseteq H$ is bounded.

\smallskip
The hypotheses on the multivalued perturbation $F(t,x)$ are:

\smallskip
$H(F)_1$: $F:T\times H\rightarrow P_{f_c}(H)$ is a multifunction such that
\begin{itemize}
	\item [(i)] for every $x\in H,\ t\mapsto F(t,x)$ is graph measurable;
	\item [(ii)] for almost all $t\in T,\ {\rm Gr}\,F(t, \cdot)\subseteq H\times H_w$ is sequentially closed (by $H_w$ we denote the Hilbert space $H$ furnished with the weak topology);
	\item [(iii)] there exists $M>0$ such that
	\begin{equation*}
		\begin{array}{ll}
			0\leq(h,x)\ \mbox{for almost all}\ t\in T,\ \mbox{and all}\ |x|=M,\ h\in F(t,x), \\
			|F(t,x)|\leq a_M(t)\ \mbox{for almost all}\ t\in T,\ \mbox{and all}\ |x|\leq M,\ \mbox{with}\ a_M\in L^{p'}(T).
		\end{array}
	\end{equation*}
\end{itemize}

Alternatively, we may assume the following conditions on $F(t,x)$:

\smallskip
$H(F)'_1$: $F:T\times H\rightarrow P_{f_c}(H)$ is a multifunction such that hypotheses $H(F)_1'(i),(ii)$ are the same as the corresponding hypotheses $H(F)_1(i),(ii)$ and
\begin{itemize}
	\item [(iii)] $|F(t,x)|\leq k(t)[1+|x|]$ for almost all $t\in T$, and all $x\in H$, \mbox{with}\ $k\in L^{p'}(T)$.
\end{itemize}
\begin{remark}
	Hypotheses $H(F)_1(i),(ii)$ imply that for all $u\in L^\infty(T,H)$ the multifunction $u\mapsto F(t,u(t))$ admits a measurable selection. Indeed, let $\{s_n\}_{n\geq1}$ be a sequence of simple functions such that $s_n(t)\rightarrow u(t)$ as $n\rightarrow\infty$ and $|s_n(t)|\leq|u(t)|$ for almost all $t\in T$, all $n\in\NN$. Then hypothesis $H(F)_1(i)$ and the Yankov-von Neumann-Aumann selection theorem imply that there exists a measurable function  $h_n:T\rightarrow H$ such that $h_n(t)\in F(t, s_n(t))$ for almost all $t\in T$ and all $n\in\NN$. Then $\{h_n\}_{n\geq1}\subseteq L^\infty(T,H)$ is bounded and so we may assume that $h_n\xrightarrow{w} h$ in $L^1(T,H)$. Invoking Proposition 3.9 of Hu \& Papageorgiou \cite[p. 694]{4} and using hypothesis $H(F)_1(ii)$ we conclude that $h(t)\in F(t,u(t))$ for almost all $t\in T$. Hypothesis $H(F)_1(iii)$ is a multivalued version of a condition due to Hartman (see \cite{8}).
\end{remark}

Let $x_0\in H$ and $h\in L^{p'}(T,H)$ and consider the following Cauchy problem:
\begin{equation}\label{eq2}
	\left\{
		\begin{array}{ll}
			-u'(t)\in A(t,u(t)) + \partial\varphi(u(t)) + h(t)\ \mbox{for almost all}\ t\in T, \\
			u(0) = x_0.
		\end{array}
	\right\}
\end{equation}
\begin{prop}\label{prop3}
	If hypotheses $H(A),\ H(\varphi)$ or $H(A)',\ H(\varphi)'$ hold, then problem (\ref{eq2}) admits a unique solution $u_0\in W_p(T)$.
\end{prop}
\begin{proof}
	First we do the proof when $H(A)$ and $ H(\varphi)$ hold.
	
	Consider the map $L:D\subseteq L^p(T,X)\rightarrow L^{p'}(T,X^*)$ defined by
	\begin{equation*}
		L(u)=u'\ \mbox{for all}\ u\in D=\{u\in W_p(T):\ u(0)=x_0\}.
	\end{equation*}
	
	By Lemma 8.93 of Roubicek \cite[p. 289]{roubicek} we know that $L(\cdot)$ is maximal monotone and densely defined.
	
	We introduce the map $a:L^p(T,X)\rightarrow L^{p'}(T,X^*)$ and the functional $\Phi:L^p(T,H)\rightarrow\overline{\RR}=\RR\cup\{+\infty\}$ defined by
	\begin{equation*}
		\begin{array}{ll}
			a(u)(\cdot) = A(\cdot,u(\cdot))\ \mbox{for all}\ u\in L^p(T,X), \\
			\Phi(u) = \int^b_0\varphi(u(t))dt\ \mbox{for all}\ u\in L^p(T,H).
		\end{array}
	\end{equation*}
	
	Theorem 2.35 in Hu \& Papageorgiou \cite[p. 41]{5} implies that
	\begin{equation*}
		a(\cdot)\ \mbox{is}\ L\mbox{-pseudomonotone}.
	\end{equation*}
	
	Also, $\Phi\in\Gamma_0(L^p(T,H))$ and $\partial\Phi(u)=S^{p'}_{\partial\varphi(u(\cdot))}\subseteq L^{p'}(T,H) = L^p(T,H^*)$ for all $u\in L^p(T,H)$ (see Theorem 9.24 in Hu \& Papageorgiou \cite[p. 271]{4}). Moreover, hypothesis $H(\varphi)$ implies that $L^p(T,X)\subseteq D(\partial\Phi)$.
	
	We claim that the multivalued map $u\mapsto a(u)+\partial\Phi(u)$ is $L$-pseudomonotone. Evidently, this multifunction has values in $P_{wkc}(L^{p'}(T,X^*))$ and it is usc from every finite dimensional subspace of $L^p(T,X)$ into $L^{p'}(T,X^*)_w$ (see Proposition 2.23 in Hu \& Papageorgiou \cite[p. 43]{4}). Consider two sequences $\{u_n\}_{n\geq1}\subseteq W_p(T)$ and $\{g_n\}_{n\geq1}\subseteq L^{p'}(T,H)$ such that
	\begin{equation}\label{eq3}
		\begin{array}{ll}
			u_n\xrightarrow{w}u\ \mbox{in}\ W_p(T),\ g_n\xrightarrow{w}g\ \mbox{in}\ L^{p'}(T,H), \ g_n\in\partial\Phi(u_n)\ \mbox{for all}\ n\in\NN, \\
			\limsup_{n\rightarrow\infty}((a(u_n)+g_n,u_n-u))\leq0
		\end{array}	
	\end{equation}
	with $((\cdot, \cdot))$ denoting the duality brackets for the pair $(L^{p'}(T,X^*),L^p(T,X))$. Recall that $L^p(T,X)^*=L^{p'}(T,X^*)$ (see Theorem 2.2.9 in Gasinski \& Papageorgiou \cite[p. 129]{3}). So
	\begin{equation*}
		((g,h))=\int^b_0\langle g(t),h(t)\rangle dt\ \mbox{for all}\ (g,h)\in L^{p'}(T,X^*)\times L^p(T,X).
	\end{equation*}
	
	We know that $W_p(T)\hookrightarrow L^p(T,H)$ compactly. Therefore we have
	\begin{equation}\label{eq4}
		u_n\rightarrow u\ \mbox{in}\ L^p(T,H)\ \mbox{(see (\ref{eq3}))}.
	\end{equation}
	
	On the other hand, $\partial\Phi(\cdot)$ is maximal monotone and so ${\rm Gr}\,\partial\Phi$ is sequently closed in $L^p(T,H)\times L^{p'}(T,H)_w$. Then it follows  from (\ref{eq3}) and (\ref{eq4}) that
	\begin{equation}\label{eq5}
		(u,g)\in {\rm Gr}\,\partial\Phi.
	\end{equation}
	
	Also, we have
	\begin{eqnarray}\label{eq6}
		&&((g_n,u_n-u))=\int^b_0\langle g_n(t), u_n(t)-u(t)\rangle dt = \int^b_0(g_n(t), u_n(t)-u(t))dt\rightarrow 0\nonumber\\
		&& \mbox{as}\ n\rightarrow\infty\ \mbox{(see (\ref{eq4}))}.
	\end{eqnarray}
	
	Returning to the last convergence in (\ref{eq3}) and using (\ref{eq6}), we obtain
	\begin{equation}\label{eq7}
		\limsup_{n\rightarrow\infty}((a(u_n),u_n-u))\leq0.
	\end{equation}
	
	But recall that $a(\cdot)$ is $L$-pseudomonotone. So, from (\ref{eq7}) we infer that
	\begin{equation}\label{eq8}
		a(u_n)\xrightarrow{w}a(u)\ \mbox{in}\ L^{p'}(T,X^*)\ \mbox{and}\ ((a(u_n),u_n))\rightarrow ((a(u),u)).
	\end{equation}
	
	Then from (\ref{eq5}), (\ref{eq6}) and (\ref{eq11}), we conclude that
	\begin{equation}\label{eq9}
		u\mapsto a(u)+\partial\Phi(u)\ \mbox{is}\ L\mbox{-pseudomonotone}.
	\end{equation}
	
	For every $u\in L^{p}(T,X)$ and every $g\in\partial\Phi(u)$, we have
	\begin{equation}\label{eq10}
		((a(u)+g,u)) = \int^b_0\langle A(t, u(t)),u(t)\rangle dt + \int^b_0(g(t),u(t))dt.
	\end{equation}
	
	Hypothesis $H(A)(ii)$ implies that
	\begin{equation}\label{eq11}
		c_0||u||^p_{L^p(T,X)}\leq\int^b_0\langle A(t,u(t)),u(t)\rangle dt.
	\end{equation}
	
	Also, $g\in \partial\Phi(u)$ implies that
	\begin{eqnarray}\label{eq12}
		& g(t)\in\partial\varphi(u(t)) & \mbox{for almost all}\ t\in T, \nonumber\\
	\quad	\Rightarrow & (g(t),u(t)) & = (g(t)-h,u(t)) + (h,u(t))\ \mbox{for all}\ h\in\partial\varphi(0) \nonumber\\
		& & \geq (h,u(t))\ \mbox{(since $\partial\varphi(\cdot)$ is monotone)}, \nonumber\\
		\Rightarrow & \int^b_0(g(t),u(t))dt & \geq 	-||u||_{L^p(T,H)}|\partial\Phi(0)| \nonumber\\
		& & \geq -c_1||u||_{L^p(T,X)}|\partial\Phi(0)|\ \mbox{for some}\ c_1>0\ \mbox{(recall that $X\hookrightarrow H$)} \nonumber\\
		& & \geq	 -c_2||u||_{L^p(T,X)}|\ \mbox{for some}\ c_2>0\ \mbox{(see hypothesis $H(\varphi)$)}.
	\end{eqnarray}
	
	We return to (\ref{eq10}) and use (\ref{eq11}), (\ref{eq12}). Then
	\begin{eqnarray}
		& ((a(u)+g,u))\geq c_0||u||^p_{L^p(T,X)}-c_2||u||_{L^p(T,X)}, \nonumber\\
		\Rightarrow & u\mapsto a(u) + \partial\Phi(u)\ \mbox{is strongly coercive.} \label{eq13}
	\end{eqnarray}
	
	Then (\ref{eq9}) and (\ref{eq13}) permit the use of Proposition \ref{prop2} and so
	\begin{equation*}
		R(L+a+\partial\Phi)=L^{p'}(T,X^*).
	\end{equation*}
	
	Therefore we can find $u_0\in W_p(T)$ such that
	\begin{equation*}
		-u'_0\in a(u_0)+\partial\Phi(u_0)+h.
	\end{equation*}
	
	Next, we show that this solution is unique. To this end, suppose that $v_0\in W_p(T)$ is another solution of problem (\ref{eq2}). We have
	\begin{eqnarray}
		&&-u_0'(t) = A(t,u_0(t)) + g_{u_0}(t) + h(t)\ \mbox{for almost all}\ t\in T,\quad
		 u_0(0)=x_0,\ g_{u_0}\in\partial\Phi(u_0), \label{eq14} \\
		&&-v_0'(t) = A(t,v_0(t)) + g_{v_0}(t) + h(t)\ \mbox{for almost all}\ t\in T,\quad
		v_0(0)=x_0,\ g_{v_0}\in\partial\Phi(v_0). \label{eq15}
	\end{eqnarray}
	
	We subtract (\ref{eq15}) from (\ref{eq14}) and obtain
	\begin{equation}\label{eq16}
		u_0'(t)-v_0'(t) + A(t,u_0(t)) - A(t,v_0(t)) + g_{v_0}(t) - g_{v_0}(t) = 0\ \mbox{for almost all}\ t\in T.
	\end{equation}
	
	On (\ref{eq16}) we act with $u_0(t)-v_0(t)\in X$ and then integrate. Using the integration by parts formula (see Proposition \ref{prop1}), the monotonicity of $A(t,\cdot)$ and hypothesis $H(\varphi)$, we have
	\begin{eqnarray*}
		&& |u_0(t)-v_0(t)|^2\leq -c_0\int^t_0|u_0(s)-v_0(s)|^2ds\leq0\ \mbox{for all}\ t\in T, \\
		&\Rightarrow & u_0=v_0.
	\end{eqnarray*}
	
	This proves the uniqueness of the solution $u_0\in W_p(T)$ of problem (\ref{eq2}).
	
	Now suppose that hypotheses $H(A)'$ and $ H(\varphi)'$ hold. The existence part of the above proof remains unchanged. For the uniqueness part, the only change is that now we have
	\begin{eqnarray*}
		&& |u_0(t)-v_0(t)|^2\leq-c_0\int^t_0|u_0(s)-v_0(s)|^2ds\leq0\ \mbox{(see hypothesis $H(A)'(ii)$)} \\
		&\Rightarrow & u_0=v_0.
	\end{eqnarray*}
The proof is now complete.
\end{proof}

We can introduce the Poincar\'e map $K:H\rightarrow H$ defined by
\begin{equation*}
	K(x_0) = u(b),
\end{equation*}
where $u\in W_p(T)$ is the unique solution of (\ref{eq2}) (see Proposition \ref{prop3}).

\begin{prop}\label{prop4}
	If hypotheses $H(A),\ H(\varphi)$ of $H(A)',\ H(\varphi')$ hold, then $K(\cdot)$ is a contraction.
\end{prop}
\begin{proof}
	Let $x_0,\hat{x}_0\in H$ be two distinct initial conditions for problem (\ref{eq2}) and let $u_0,\hat{u}\in W_p(T)$ be the corresponding unique solutions of the Cauchy problem (\ref{eq2}) (see Proposition \ref{prop3}). We have
	\begin{eqnarray}
		-u_0'(t)\in A(t,u_0(t)) + \partial\varphi(u_0(t)) + h(t)\ \mbox{for almost all}\ t\in T, \ u_0(0)=x_0, \label{eq17}\\
		-v_0'(t)\in A(t,v_0(t)) + \partial\varphi(v_0(t)) + h(t)\ \mbox{for almost all}\ t\in T, \ v_0(0)=x_0. \label{eq18}
	\end{eqnarray}
	
	First we assume that hypotheses $H(A)$ and $ H(\varphi)$ hold.
	
	As before, subtracting (\ref{eq18}) from (\ref{eq17}) and using Proposition \ref{prop1} and hypothesis $H(\varphi)$, we obtain
	\begin{eqnarray*}
		&& \frac{1}{2}\frac{d}{dt}|u_0(t)-\hat{u}(t)|^2\leq-c_0|u_0(t)-\hat{u}(t)|^2\ \mbox{for almost all}\ t\in T, \\
		&\Rightarrow & \frac{d}{dt}\left[e^{2c_0t}|u_0(t)-\hat{u}(t)|^2\right]\leq0\ \mbox{for almost all}\ t\in T, \\
		&\Rightarrow & |u_0(t)-\hat{u}(t)|\leq e^{-2c_0t}|x_0-\hat{x}|\ \mbox{for all}\ t\in T.
	\end{eqnarray*}
	
	It follows that
	\begin{eqnarray*}
		&& |K(x_0)-K(\hat{x})|\leq e^{-2c_0b}|x_0-\hat{x}|, \\
		&\Rightarrow & K(\cdot)\ \mbox{is a contraction}.
	\end{eqnarray*}
	
	If hypotheses $H(A)'$ and $ H(\varphi)'$ hold, then
	\begin{eqnarray*}
		\frac{1}{2}\frac{d}{dt}|u_0(t)-\hat{u}(t)|^2 & \leq & -c_0||u_0(t)-\hat{u}(t)||^2 \\
		& \leq & -c_3|u_0(t)-\hat{u}(t)|^2\ \mbox{for almost all}\ t\in T\\
		& & \mbox{and some}\ c_3>0\ \mbox{(recall that $X\hookrightarrow H$)}
	\end{eqnarray*}
	and then continuing as above, we obtain
	\begin{equation*}
		|K(x_0)-K(\hat{x})|\leq e^{-2c_3b}|x_0-\hat{x}|.
	\end{equation*}
The proof is now complete.
\end{proof}

Given $h\in L^{p'}(T,H)$, we consider the following periodic problem:
\begin{equation}\label{eq19}
	\left\{
		\begin{array}{l}
			-u'(t)\in A(t,u(t)) + \partial\varphi(u(t)) + h(t)\ \mbox{for almost all}\ t\in T,\\
			u(0) = u(b).
		\end{array}
	\right\}
\end{equation}
\begin{prop}\label{prop5}
	If hypotheses $H(A),\ H(\varphi)$ or $H(A)',\ H(\varphi)'$ hold, then problem (\ref{eq19}) has a unique solution
	\begin{equation*}
		u_0\in W_p(T)\subseteq C(T,H)
	\end{equation*}
	and we have
	\begin{equation}\label{eq20}
		|u_0(t)|\leq \hat{c} + \int^t_0|h(s)|ds\ \mbox{for all}\ t\in T,\ \mbox{and some}\ \hat{c}>0.
	\end{equation}
\end{prop}
\begin{proof}
	By Proposition \ref{prop4}, we know that for both cases the Poincar\'e map $K:H\rightarrow H$ is a contraction. So, the Banach fixed point theorem guarantees the existence of a unique $x_0\in H$ such that
	\begin{equation}\label{eq21}
		K(x_0)=x_0.
	\end{equation}
	
	Let $u_0\in W_p(T)\subseteq C(T,H)$ be the unique solution of (\ref{eq2}) with $u_0(0)=x_0\in H$.	
	From (\ref{eq21}) it follows that this is the unique solution of (\ref{eq21}).
	
	Next, we establish the uniform bound in (\ref{eq20}). We have
	\begin{equation*}
		\left\{
			\begin{array}{l}
				-u'(t)\in A(t,u_0(t)) + \partial\varphi(u_0(t)) + h(t)\ \mbox{for almost all}\ t\in T,\\
				u_0(0) = u_0(b).
			\end{array}
		\right\}
	\end{equation*}
	
	The proof is common for both cases. We act with $u_0(t)$ and use Proposition \ref{prop1}. Then for some $g_0\in S^{p'}_{\partial\varphi(u_0(\cdot))}$ and for all $\eta\in\partial\varphi(0)$, we have
	\begin{eqnarray}
		&\frac{1}{2}\frac{d}{dt}  |u_0(t)|^2 &\leq-c_4|u_0(t)|^2  -(g_0(t)-\eta,u_0(t)) - (\eta,u_0(t))-(h(t),u_0(t)) \nonumber\\
		& & \mbox{for almost all}\ t\in T,\ \mbox{and some}\ c_4>0\ \mbox{(see hypothesis $H(A)(ii)$)} \nonumber\\
		& & \leq -c_4|u_0(t)|^2 + [|\partial\varphi(0)|+|h(t)|]|u_0(t)|\ \mbox{for almost all}\ t\in T \nonumber\\
		& & \mbox{(since $\partial\varphi(\cdot)$ is monotone)} \nonumber\\
		\Rightarrow & |u_0(t)|\frac{d}{dt}|u_0(t)| & \leq-c_4|u_0(t)|^2+[|\partial\varphi(0)|+|h(t)|]|u_0(t)|\ \mbox{for almost all}\ t\in T, \nonumber\\
		\Rightarrow & \frac{d}{dt}|u_0(t)| & \leq-c_4|u_0(t)| + [|\partial\varphi(0)|+|h(t)|]\ \mbox{for almost all}\ t\in T, \nonumber\\
		\Rightarrow & \frac{d}{dt}[e^{c_4t}|u_0(t)|] & \leq e^{c_4t}[|\partial\varphi(0)|+|h(t)|] \nonumber\\
		\Rightarrow & |u_0(t)|\leq e^{-c_4t}|u_0(0)| & +\ e^{-c_4t}\int^t_0 e^{c_4s}[|\partial\varphi(0)|+|h(t)|]ds \nonumber\\
		& & \leq e^{-c_4t}|u_0(0)|+|\partial\varphi(0)|b+\int^t_0|h(s)|ds\ \mbox{for all}\ t\in T. \label{eq22}
	\end{eqnarray}
	
	If $t=b$, then using the periodic boundary condition, we have
	\begin{eqnarray}
		& \frac{e^{c_4b-1}}{e^{c_4b}}\,|u_0(0)|\leq |\partial\varphi(0)|b+||h||_{L^1(T,H)}, \nonumber\\
		\Rightarrow & |u_0(0)|\leq\frac{e^{c_4b}}{e^{c_4b-1}}\left[|\partial\varphi(0)|b+||h||_{L^1(T,H)}|\right]. \label{eq23}
	\end{eqnarray}
	
	We return to (\ref{eq22}) and use (\ref{eq23}). Then
	\begin{eqnarray*}
		& |u_0(t)|\leq\frac{e^{c_4b}}{e^{c_4b}-1}\left[|\partial\varphi(0)|b+||h||_{L^1(T,H)}\right]+|\partial\varphi(0)|b+\int^t_0|h(s)|ds \\
		\Rightarrow & |u_0(t)|\leq\hat{c}+\int^t_0|h(s)|ds\ \mbox{for all}\ t\in T,\ \mbox{and some}\ \hat{c}>0.
	\end{eqnarray*}
The proof is now complete.
\end{proof}
Let $M>0$ be as in hypothesis $H(F)(iii)$ and let $p_M:H\rightarrow H$ be the $M$-radial retraction defined by
\begin{equation*}
	p_M(x)=\left\{
     \begin{array}{ll}
       x & \mbox{if}\ |x|\leq M\\
       \frac{Mx}{|x|} & \mbox{if}\ M<|x| \\
     \end{array}
   \right. \mbox{for all}\ x\in H.
\end{equation*}

We set $\hat{F}(t,x)=F(t,p_M(x))$ for all $(t,x)\in T\times H$. Clearly, $\hat{F}(t,x)$ satisfies hypotheses $H(F),(i),(ii)$ and
\begin{eqnarray*}
	&&|\hat{F}(t,x)|\leq a_M(t)\ \mbox{for almost all}\ t\in T,\ \mbox{and all}\ x\in H,\ \mbox{where}\ a_M\in L^{p'}(T)\\
	&&\mbox{(see hypothesis $H(F),(iii)$)}.
\end{eqnarray*}

In what follows, we denote by $\hat{S}\subseteq W_p(T)$ the solution set of (\ref{eq1}) with $F$ replaced by $\hat{F}$, and by $S\subseteq W_p(T)$ the solution set of (\ref{eq1}) with the original $F$.

Also, we will need the following extra condition on $\varphi$.

\smallskip
$H_0$: For all $(u,h)\in {\rm Gr}\,\partial\varphi$, we have $0\leq(h,u)$.
\begin{remark}
	Evidently, this condition is satisfied if $0\in\partial\varphi(0)$ (hence $0$ is a minimizer of $\varphi$).
\end{remark}
\begin{prop}\label{prop6}
	\begin{itemize}
		\item [(a)] If hypotheses $H(A),\ H(\varphi)$ or $H(A'),\ H(\varphi)'$ and $H(F)_1,\ H_0$ hold, then $|u(t)|\leq M$ for all $t\in T$,  $u\in\hat{S}$.
		\item [(b)] If hypotheses $H(A),\ H(\varphi)$ or $H(A)',\ H(\varphi)'$ and $H(F)'_1$ hold, then there exists $M>0$ such that $|u(t)|\leq M$ for all $t\in T$, $u\in S$.
	\end{itemize}
\end{prop}
\begin{proof}
	(a) Suppose that the conclusion of this part is not true. Then for some $u\in \hat{S}$ one of the following assertions holds.
	\begin{itemize}
		\item $|u(t)|>M$ for all $t\in T$.
		\item There exist $\tau,r\in T$ with $\tau<r$ such that $|u(\tau)|=M$ and $|u(r)|>M$.
	\end{itemize}
	We have
	\begin{equation*}
		-u'(t)\in A(t,u(t)) + \partial\varphi(u(t)) + \hat{h}(t)\ \mbox{a.e on}\ T,\ u(0)=u(b)
	\end{equation*}
	with $\hat{h}\in S^{p'}_{\hat{F}(\cdot,u(\cdot))}$. As before, using Proposition \ref{prop1} and hypotheses $H(A)(iii)$ and $ H_0$ we have
	\begin{eqnarray*}
		& |u(t)|^2 + c_0\int^t_0||u(s)||^2ds & \leq |u(0)|^2-\int^t_0(\hat{h}(s),u(s))ds \\
		& & =|u(0)|^2-\int^t_0\frac{u(s)}{M}(\hat{h}(s),p_M(u(s)))ds \\\
		& & \leq |u(0)|^2\ \mbox{(see hypothesis $H(F)_1(iii)$)}, \\
		\Rightarrow & |u(b)|^2 < |u(0)|^2,\ \mbox{a contradiction}. &
	\end{eqnarray*}

	If the second case holds, then repeating the above argument on the interval $[\tau,r]$, we obtain
	\begin{equation*}
		|u(r)|^2<|u(\tau)|^2,\ \mbox{again a contradiction}.
	\end{equation*}

	Therefore we conclude that
	\begin{equation*}
		|u(t)|\leq M\ \mbox{for all}\ t\in T,\  u\in\hat{S}.
	\end{equation*}

	(b) Let $u\in S\subseteq W_p(T)$. Then we have
	\begin{eqnarray*}
		& - & u'(t)\in A(t,u(t)) + \partial\varphi(u(t)) + h(t)\ \mbox{for almost all}\ t\in T, \\
		& & u(0) = u(b),
	\end{eqnarray*}
	with $h\in S^{p'}_{F(\cdot,u(\cdot))}$. Then form (\ref{eq20}) of Proposition \ref{prop5}, we have
	\begin{eqnarray*}
		& |u(t)| & \leq \hat{c} + \int^t_0|h(s)|ds \\
		& & \leq \hat{c} + \int^t_0 k(s)[1+|u(s)|]ds\ \mbox{(see hypothesis $H(F)_1'(iii)$)} \\
		\Rightarrow & |u(t)| \leq& M\ \mbox{for some}\ M>0,\ \mbox{and all}\ t\in T,\ u\in S\ \mbox{(use Gronwall's inequality)}.
	\end{eqnarray*}
This completes the proof.
\end{proof}
On account of Proposition \ref{prop6}, we see that we can replace $F(t,x)$ by
\begin{equation*}
	\hat{F}(t,x) = F(t,p_M(x))\ \mbox{for all}\ (t,x)\in T\times H.
\end{equation*}

As we have already mentioned, $\hat{F}$ preserves the properties of $F$. More precisely, we have:
\begin{itemize}
	\item For all $x\in H,\ t\mapsto\hat{F}(t,x)$ is graph measurable.
	\item For almost all $t\in T,\ {\rm Gr}\,\hat{F}(t,\cdot)\subseteq H\times H_w$ is sequentially closed.
\end{itemize}

Moreover, we have
\begin{equation*}
	|\hat{F}(t,x)|\leq\hat{\eta}(t)\ \mbox{for almost all}\ t\in T,\ \mbox{and all}\ x\in H
\end{equation*}
with $\hat{\eta}\in L^{p'}(T)$ ($\hat{\eta}=a_M$ if $H(F)_1,H_0$ hold and $\hat{\eta}=(1+M)k$ if $H(F)_2$ holds).

Let $\xi:L^{p'}(T)\rightarrow C(T,H)$ be the solution map for problem (\ref{eq19}). So, for every $h\in L^{p'}(T,H)$, $\xi(h)\in W_p(T)\subseteq C(T,H)$ is the unique solution of problem (\ref{eq19}) (see Proposition \ref{prop5}).
\begin{prop}\label{prop7}
	If hypotheses $H(A),\ H(\varphi)$ or $H(A)',\ H(\varphi)'$ hold, then $\xi:L^{p'}(T,H)\rightarrow C(T,H)$ is completely continuous (that is, if $h_n\xrightarrow{w}h$ in $L^{p'}(T,H)$, then $\xi(h_n)\rightarrow\xi(h)$ in $C(T,H)$).
\end{prop}
\begin{proof}
	Suppose that $h_n\xrightarrow{w}h$ in $L^{p'}(T,H)$ and let $u_n=\xi(h_n)$ for all $n\in\NN, u=\xi(h)$. Then there exist $\{g_n,g\}_{n\geq1}\subseteq L^{p'}(T,H)$ such that
	\begin{eqnarray}\label{eq24}
		&&g_n(t)\in\partial\varphi(u_n(t))\ \mbox{for almost all}\ t\in T,\ \mbox{and all}\ n\in\NN,\ \mbox{and}\ g(t)\in\partial\varphi(u(t))\nonumber\\
		&& \mbox{for almost all}\ t\in T.
	\end{eqnarray}

	We have
	\begin{eqnarray}
		u'_n(t) + A(t,u_n(t)) + g_n(t) + h_n(t) = 0,\ \mbox{for almost all}\ t\in T,\ \mbox{and all}\ n\in\NN,\nonumber\\  u_n(0)=u_n(b). \label{eq25}\\
		u'(t) + A(t, u(t)) + g(t) + h(t) = 0\ \mbox{for almost all}\ t\in T,\ u(0)=u(b). \label{eq26}
	\end{eqnarray}

	On (\ref{eq25}) we act with $u_n(t)$. Using Proposition \ref{prop1} (the integration by parts formula) and hypothesis $H(A)(ii)$ or $H(A)'(ii)$ we obtain
	\begin{equation*}
		\frac{1}{2}\frac{d}{dt}|u_n(t)|^2 + c_0||u_n(t)||^p + (g_n(t),u_n(t)) + (h_n(t),u_n(t)) = 0\ \mbox{for almost all}\ t\in T,
	\end{equation*}

	Integrating over $T$ and using the periodic boundary condition, we have
	\begin{eqnarray*}
		c_0||u_n||^p_{L^p(T,X)} + \int^b_0(g_n(t) - v^*,u_n(t))dt + \int^b_0(v^*,u_n(t))dt \leq c_5||u_n||_{L^p(T,X)} \\
		\mbox{for some}\ c_5>0,\ \mbox{and all}\ n\in\NN\ \mbox{with}\ v^*\in\partial\varphi(0).
	\end{eqnarray*}

	Hypothesis $H(\varphi)$ or $H(\varphi)'$ implies that
	\begin{eqnarray}
		& & c_0||u_n||^p_{L^p(T,X)}\leq c_6||u_n||_{L^p(T,X)}\ \mbox{for some}\ c_6>0,\ \mbox{and all}\ n\in\NN, \nonumber\\
		& \Rightarrow & \{u_n\}_{n\geq1}\subseteq L^p(T,X)\ \mbox{is bounded}. \label{eq27}
	\end{eqnarray}

	Recall that $\partial\varphi(\cdot)$ is bounded (see hypothesis $H(\varphi)$ and \cite{4}). Hence, if $M>0$ is as in Proposition \ref{prop6} and $\overline{B}_M=\{x\in H:|x|\leq M\}$, then
	\begin{equation*}
		\partial\varphi(\overline{B}_M)\subseteq H\ \mbox{is bounded}.
	\end{equation*}

	So, we can find $M_1>0$ such that
	\begin{equation}\label{eq28}
		|\partial\varphi(u_n(t))|\leq M_1\ \mbox{for all}\ n\in\NN,\ t\in T.
	\end{equation}

	From (\ref{eq25}), (\ref{eq27}), (\ref{eq28}) and hypothesis $H(A)(iii)$ it follows that
	\begin{equation}\label{eq29}
		||u'_n||_{L^{p'}(T,X^*)}\leq M_2\ \mbox{for some}\ M_2>0,\ \mbox{and all}\ n\in\NN.
	\end{equation}

	From (\ref{eq27}) and (\ref{eq29}) we infer that
	\begin{equation}\label{eq30}
		\{u_n\}_{n\geq1}\subseteq W_p(T)\ \mbox{is bounded}.
	\end{equation}

	So by passing to a suitable subsequence if necessary, we may assume that
	\begin{eqnarray}
		&&u_n\rightarrow\hat{u}\ \mbox{in}\ L^p(T,H)\ \mbox{and}\ u_n\xrightarrow{w}\hat{u}\ \mbox{in}\ C(T,H) \label{eq31} \\
		&&\mbox{(recall that $W_p(T)\hookrightarrow L^p(T,H)$ compactly, $W_p(T)\hookrightarrow C(T,H)$ and see (\ref{eq30}))}. \nonumber
	\end{eqnarray}

	Let $\epsilon_m\rightarrow0^+$ be such that for all $m\in\NN,\ u_n(\epsilon_m)\rightarrow\hat{u}(\epsilon_m)$ in $H$ as $n\rightarrow\infty$ (see the first convergence in (\ref{eq31})). As before, using (\ref{eq25}), (\ref{eq26}) and Proposition \ref{prop1} (the integration by parts formula), we have for all $n,\, m\in\NN$ and all $t\in[\epsilon_m,b]$
	\begin{eqnarray}
		|u_n(t)-u(t)|^2 & \leq |u_n(\epsilon_m)-u(\epsilon_m)|^2 + \int^t_{\epsilon_m}\langle A(s,u_n(s)) - A(s,u(s)), u(s)-u_n(s)\rangle ds \nonumber \\
		& +\int^t_{\epsilon_m}(g_n(s)-g(s),u(s)-u_n(s)) ds + \int^t_{\epsilon_m}(h_n(s)-h(s),u(s)-u_n(s))ds \nonumber\\
		& \leq |u_n(\epsilon_m)-u(\epsilon_m)|^2 - c_7\int^t_{\epsilon_m}|u_n(s)-u(s)|^2ds + \label{eq32}\\
		& \int^t_{\epsilon_m}(h_n(s)-h(s),u(s)-u_n(s))ds\ \mbox{for some}\ c_7>0,\ \mbox{and all}\ n\in\NN. \nonumber
	\end{eqnarray}

	To derive (\ref{eq32}) if $H(A)$ and $ H(\varphi)$ hold, we have used $H(A)(ii)$, the monotonicity of $\partial\varphi(\cdot)$ and the fact that $X\hookrightarrow H$, while if $H(A)',\ H(\varphi)'$ hold, we have used the strong monotonicity of $\partial\varphi(\cdot)$ and the monotonicity of $A(t,\cdot)$.

	In (\ref{eq32}) we pass to the limit as $n\rightarrow\infty$ and obtain
	\begin{eqnarray*}
		|\hat{u}(t)-u(t)|^2 \leq |\hat{u}(\epsilon_m)-u(\epsilon_m)|^2 - c_7\int^t_{\epsilon_m}|\hat{u}(s)-u(s)|^2 ds\ \mbox{for all}\ m\in\NN \\
		\mbox{(see (\ref{eq31}) and recall that $u_n(\epsilon_m)\rightarrow\hat{u}(\epsilon_m)$ in $H$ for all $m\in\NN$ and $2\leq p$)}.
	\end{eqnarray*}
	
	Finally, letting $m\rightarrow\infty$ we get
	\begin{equation*}
		|\hat{u}(t)-u(t)|^2\leq|\hat{u}(0)-u(0)|^2 - c_7\int^t_0|\hat{u}(s)-u(s)|^2 ds\ \mbox{for all}\ t\in T.
	\end{equation*}
	
	Choosing $t=b$ and recalling that $\hat{u}(0)=\hat{u}(b), u(0)=u(b)$, we have
\begin{eqnarray*}
	&&0\leq -c_7\int^b_0|\hat{u}(s)-u(s)|^2ds\leq 0\\
	&\Rightarrow&\hat{u}=u.
\end{eqnarray*}

It follows from (\ref{eq32}) that
$$||u_n-u||_{C(T,H)}\rightarrow 0\ \mbox{as}\ n\rightarrow\infty.$$

Hence for the original sequence we have
\begin{eqnarray*}
	&&u_n=\xi(h_n)\rightarrow\xi(h)=u\ \mbox{in}\ C(T,H)\ \mbox{as}\ n\rightarrow\infty,\\
	&\Rightarrow&\xi:L^{p'}(T,H)\rightarrow C(T,H)\ \mbox{is completely continuous.}
\end{eqnarray*}
The proof is complete.
\end{proof}

As we have already indicated by replacing $F$ by $\hat{F}$ if necessary, we may assume that
$$|F(t,x)|\leq\hat{\eta}(t)\ \mbox{for almost all}\ t\in T,\ \mbox{and all}\ x\in H\ \mbox{with}\ \hat{\eta}\in L^{p'}(T).$$

Based on this, we introduce the following set
$$W=\{h\in L^{p'}(T,H):|h(t)|\leq\hat{\eta}(t)\ \mbox{for almost all}\ t\in T\}.$$

From the Eberlein-Smulian theorem we know that $W\subseteq L^{p'}(T,H)$ is sequentially weakly compact. Therefore, using Proposition \ref{prop7}, we conclude that
\begin{equation}\label{eq33}
	E=\xi(W)\subseteq C(T,H)\ \mbox{is compact.}
\end{equation}

Now we are ready for our first existence theorem for the ``convex" problem (that is, the multivalued perturbation $F(t,x)$ is convex-valued).
\begin{theorem}\label{th8}
	If hypotheses $H(A),\ H(\varphi)$ or $H(A)',H(\varphi)'$ and $H(F)_1,\ H_0$ or $H(F)_1'$ hold, then problem (\ref{eq1}) admits a solution $\hat{u}\in W_p(T)$.
\end{theorem}
\begin{proof}
	We furnish $W\subseteq L^{p'}(T,H)$ with the relative weak topology and consider the multifunction $H:W\rightarrow P_{kc}(W)$ defined by
	$$H(h)=S^{p'}_{F(\cdot,\xi(h)(\cdot))}.$$
	
	Let $\{(h_n,g_n)\}_{n\geq 1}\subseteq {\rm Gr}\,H$ and assume that
	\begin{eqnarray}\label{eq34}
		h_n\stackrel{w}{\rightarrow}h,\ g_n\stackrel{w}{\rightarrow}g\ \mbox{in}\ L^{p'}(T,H)\ \mbox{as}\ n\rightarrow\infty.
	\end{eqnarray}
	
	Then (\ref{eq34}) and Proposition \ref{prop7} imply that
	\begin{equation}\label{eq35}
		\xi(h_n)\rightarrow\xi(h)\ \mbox{in}\ C(T,H)\ \mbox{as}\ n\rightarrow\infty.
	\end{equation}
	
	Invoking Proposition 3.9 of Hu \& Papageorgiou \cite[p. 694]{4}, we have
	\begin{eqnarray*}
		&&g(t)\in\overline{\rm conv}\, w-\limsup\limits_{n\rightarrow\infty}F(t,\xi(h_n)(t))\\
		&&\subseteq F(t,\xi(h)(t))\ \mbox{for almost all}\ t\in T\\
		&&(\mbox{see (\ref{eq35}) and hypothesis}\ H(F)_1(ii)=H(F)'_1(ii))\\
		&\Rightarrow&(h,g)\in {\rm Gr}\,H,\\
		&\Rightarrow&H(\cdot)\ \mbox{is usc (see Proposition 2.23 of Hu \& Papageorgiou \cite[p. 43]{4})}.
	\end{eqnarray*}
	
	By the Kakutani-Ky Fan fixed point theorem (see Theorem 2.6.7 in Papageorgiou, Kyritsi \& Yiallourou \cite[p. 114]{7}), $H(\cdot)$ admits a fixed point. So, there exists $h_0\in W$ such that
	\begin{eqnarray*}
		&& h_0\in H(g_0),\\
		&\Rightarrow&h_0\in S^{p'}_{F(\cdot),\xi(h_0)(\cdot)}.
	\end{eqnarray*}
	
	Let $u_0=\xi(h_0)\in W_p(T)$. Then
	\begin{eqnarray*}
		&&-u'_0(t)\in A(t,u_0(t))+\partial\varphi(u_0(t))+h_0(t)\ \mbox{for almost all}\ t\in T,\ u_0(0)=u_0(b),\\
		&\Rightarrow&u_0\in W_p(T)\ \mbox{is a solution of problem (\ref{eq1})}.
	\end{eqnarray*}
The proof of Theorem \ref{th8} is complete.
\end{proof}

The above existence theorem was proved under the assumption that at least one of $A(t,\cdot)$ and $\partial\varphi(\cdot)$ is strongly monotone (see hypotheses $H(A)'$ and $H(\varphi)$). Next, we remove this requirement.

So, the new hypotheses on $A(t,x)$ and $\partial\varphi(x)$ are the  following:

\smallskip
$H(A)_1:$ $A:T\times X\rightarrow X^*$ is a map such that
\begin{itemize}
	\item[(i)] for all $x\in X,\ t\mapsto A(t,x)$ is measurable;
	\item[(ii)] for almost all $t\in T,\ x\mapsto A(t,x)$ is demicontinuous and
	$$c_0||x||^p\leq \left\langle A(t,x),x\right\rangle\ \mbox{for almost all}\ t\in T,\ \mbox{all}\ x\in X;$$
	\item[(iii)] $||A(t,x),x||_*\leq a_1(t)+c_1||x||^{p-1}$ for almost all $t\in T$, and all $x\in X$ with $a_1\in L^{p'}(T),\ c_1>0$.
\end{itemize}

$H(\varphi)_1:$ $\varphi\in\Gamma_0(H)$ is bounded above on bounded sets, for all $u\in L^p(T,X)$ we have $S^{p'}_{\partial\varphi(u(\cdot))}\neq\emptyset$ and $\partial\varphi(0)\subseteq H$ is bounded.

\smallskip
By the Troyanski renorming theorem (see Gasinski-Papageorgiou \cite[p. 911]{3}) we may assume without any loss of generality that both $X$ and $X^*$ are locally uniformly convex.

Let $\mathcal{F}:X\rightarrow X^*$ be the duality map defined by
$$\mathcal{F}(x)=\{x^*\in X^*:\left\langle x^*,x\right\rangle=||x||^2=||x^*||^2_*\}.$$

The local uniform convexity of $X$ and $X^*$ implies that $\mathcal{F}(\cdot)$ is single-valued, bounded, monotone, bicontinuous bijection (hence maximal monotone, too), coercive and $\mathcal{F}^{-1}$ is the duality map of $X^*$ (see Gasinski \& Papageorgiou \cite{3} and Zeidler \cite{12}).

Note that, if $\psi(x)=\frac{1}{2}||x||^2$ for all $x\in X$, then
$$\mathcal{F}(x)=\partial\psi(x)\ (\mbox{see \cite[p. 132]{3}})$$
and so by Rockafellar \& Wets \cite[p. 565]{10} we have
$$\mathcal{F}(\cdot)\ \mbox{is strongly monotone}.$$

Using this observation we can prove the following existence theorem.
\begin{theorem}\label{th9}
	If hypotheses $H(A)_1,\ H(\varphi)_1$ and $H(F)_1,\ H_0$ or $H(\xi)'_1$ hold, then problem (\ref{eq1}) admits a solution $\hat{u}\in W_p(T)$.
\end{theorem}
\begin{proof}
	Let $\epsilon_n\rightarrow 0^+$ and consider the following approximating evolution inclusion
	$$\left\{\begin{array}{l}
		-u'(t)\in A(t,u(t))+\epsilon_n\mathcal{F}(u(t))+\partial\varphi(u(t))+F(t,u(t))\ \mbox{for almost all}\ t\in T,\\
		u(0)=u(b).
	\end{array}\right\}$$
	
	Note that for every $n\in\NN$ the mapping $x\mapsto A(t,x)+\epsilon_n\mathcal{F}(x)$ satisfies the strong monotonicity condition in hypothesis $H(A)(ii)$. So, by Theorem \ref{th8} we can find a solution $u_n\in W_p(T)$ ($n\in\NN$) for the periodic problem. We have $\{u_n\}_{n\geq 1}\subseteq E$ and so
	$$\{u_n\}_{n\geq 1}\subseteq C(T,H)\ \mbox{is relatively compact}.$$
	
	Also, since $|u_n(t)|\leq M$ for all $t\in T$, $n\in\NN$, it follows that
	$$\{u_n\}_{n\geq 1}\subseteq W_p(T)\ \mbox{is bounded}.$$
	
	So, by passing to a suitable subsequence if necessary, we may assume that
	$$u_n\stackrel{w}{\rightarrow}\hat{u}\ \mbox{in}\ W_p(T)\ \mbox{and}\ u_n\rightarrow\hat{u}\ \mbox{in}\ C(T,H).$$
	
	Then as in the proofs of Proposition \ref{prop3} and \ref{prop7}, taking the limit as $n\rightarrow\infty$, we have
	$$\left\{\begin{array}{l}
		-\hat{u}'(t)\in A(t,\hat{u}(t))+\partial\varphi(\hat{u}(t))+F(t,\hat{u}(t))\ \mbox{for almost all}\ t\in T,\\
		\hat{u}(0)=\hat{u}(b),
	\end{array}\right\}$$
	which shows that $\hat{u}\in W_p(T)$ is a solution of (\ref{eq1}).
\end{proof}

Let $\hat{S}_c\subseteq W_p(T)\subseteq C(T,H)$  denote the solution set of ``convex" problem. Then we have the following property.
\begin{theorem}\label{th10}
	If hypotheses $H(A)_1,\ H(\varphi)_1$ and $H(F)_1,\ H_0$ or $H(F)'_1$ hold, then $\hat{S}_c\in P_k(C(t,H))$.
\end{theorem}

\section{The ``nonconvex" problem}

In this section we consider problem (\ref{eq1}) when the multivalued perturbation $F(t,x)$ has nonconvex values.

Now, the hypotheses on $F(t,x)$ are the following:

$H(F)_2:$ $F:T\times H\rightarrow P_f(H)$ is a multifunction such that

\smallskip
\begin{itemize}
	\item[(i)] the mapping $(t,x)\mapsto F(t,x)$ is graph measurable;
	\item[(ii)] for almost all $t\in T,\ x\mapsto F(t,x)$ is lsc;
	\item[(iii)] there exists $M>0$ such that
	\begin{eqnarray*}
		&&0\leq(h,x)\ \mbox{for almost all}\ t\in T,\ \mbox{and all}\ |x|=M,\ h\in F(t,x),\\
		&&|F(t,x)|\leq a_M(t)\ \mbox{for almost all}\ t\in T,\ \mbox{and all}\ |x|\leq M,\ \mbox{with}\ a_M\in L^{p'}(T).
	\end{eqnarray*}
\end{itemize}

Alternatively, we can assume the following:

\smallskip
$H(F)'_2:$ $F:T\times H\rightarrow P_f(H)$ is a multifunction such that hypotheses $H(F)'_2(i),\, (ii)$ are the same as the corresponding hypotheses $H(F)_2(i),\, (ii)$ and
\begin{itemize}
	\item[(iii)] $|F(t,x)|\leq k(t)[1+|x|]$ for almost all $t\in T$, and all $x\in H$, with $k\in L^{p'}(T)$.
\end{itemize}

Following the approach of the previous section, we first consider problem (\ref{eq1}) under the strong monotonicity conditions.
\begin{theorem}\label{th11}
	If hypotheses $H(A),\ H(\varphi)$ or $H(A)',\ H(\varphi)'$ and $H(F)_2,\ H_0$ or $H(F)'_2$ hold, then problem (\ref{eq1}) admits a solution $\hat{u}\in W_p(T).$
\end{theorem}
\begin{proof}
	Clearly, Proposition \ref{prop6} can also be applied here. So, by replacing $F(t,x)$ with $\hat{F}(t,x)$, we may assume that
	\begin{equation}\label{eq36}
		|F(t,x)|\leq\hat{\eta}(t)\ \mbox{for almost all}\ t\in T,\ \mbox{and all}\ x\in H,\ \mbox{with}\ \hat{\eta}\in L^{p'}(T).
	\end{equation}
	
	As before, we introduce the set
	$$W=\{h\in L^{p'}(T,H):|h(t)|\leq\hat{\eta}(t)\ \mbox{for almost all}\ t\in T\}.$$
	
	On account of Proposition \ref{prop7}, we have
	$$\xi=\xi(W)\in P_k(C(T,H)).$$
	
	Let $E_*=\overline{\rm conv}\,E\in P_{kc}(C(T,H))$ and consider the multifunction $H:E_*\rightarrow P_{wk}(L^{p'}(T,H))$ defined by
	$$H(u)=S^{p'}_{F(\cdot,u(\cdot))}\ \mbox{for all}\ u\in E_*\ (\mbox{see (\ref{eq36})}).$$
	
	We claim that $H(\cdot)$ is lsc. According to Proposition 2.6 of Hu \& Papageorgiou \cite[p. 37]{4}, to show the lower semicontinuity of $H(\cdot)$, it suffices to prove that
$$\mbox{if $u_n\rightarrow u$ in $C(T,H)$, then $H(u)\leq\liminf_{n\rightarrow\infty}H(u_n)$}.$$

	Let $h\in V(u)$ and for every $n\in\NN$ consider the multifunction $G_n:T\rightarrow P_{wk}(H)$ defined by
	$$G_n(t)=\left\{v\in F(t,u_n(t)):|h(t)-v|\leq d(h(t),F(t,u_n(t)))+\frac{1}{n}\right\}.$$
	
	Hypothesis $H(F)_2(i)$ implies that the mapping $t\mapsto F(t,u_n(t))$ is measurable for every $n\in\NN$. It follows that
	$${\rm Gr}\,G_n\in\mathcal{L}_T\otimes B(H),$$
	with $\mathcal{L}_T$ being the Lebesgue $\sigma$-field of $T$ and $B(H)$ the Borel $\sigma$-field of $H$. Invoking the Yankov-von Neumann-Aumann selection theorem, we can find a measurable function $h_n:T\rightarrow H$ such that
	\begin{eqnarray*}
		&&h_n(t)\in G_n(t)\ \mbox{for almost}\ t\in T,\\
		&\Rightarrow&|h(t)-h_n(t)|\leq d(h(t),F(t,u_n(t)))+\frac{1}{n}\ \mbox{for almost all}\ t\in T,\ \mbox{and all}\ n\in\NN,\\
		&\Rightarrow&\limsup\limits_{n\rightarrow\infty}|h(t)-h_n(t)|\leq\limsup\limits_{n\rightarrow\infty}d(h(t),F(t,u_n(t)))\\
		&&\leq d(h(t),\liminf_{n\rightarrow\infty}F(t,u_n(t)))\\
		&&(\mbox{see Proposition 1.47 in Hu \& Papageorgiou \cite[p. 672]{4}})\\
		&&\leq d(h(t),F(t,u(t)))\ \mbox{for almost all}\ t\in T,\\
		&&(\mbox{see (\ref{eq37}) and hypothesis}\ H(F)_2(ii)),\\
		&\Rightarrow&h_n(t)\rightarrow h(t)\ \mbox{in}\ H\ \mbox{for almost all}\ t\in T,\\
		&\Rightarrow&h_n\rightarrow h\ \mbox{in}\ L^{p'}(T,H)\ \mbox{and}\ h_n\in H(u_n)=S^{p'}_{F(\cdot,u_n(\cdot))}\ \mbox{for all}\ n\in\NN,\\
		&\Rightarrow&H(\cdot)\ \mbox{is lsc}.
	\end{eqnarray*}
	
	Also, $H(\cdot)$ has decomposable values. So, we can apply the Bressan-Colombo  selection theorem \cite{1} and find  a continuous map $v:E_*\rightarrow L^{p'}(T,H)$ such that
	$$v(u)\in H(u)\ \mbox{for all}\ u\in E_*.$$
	
	We define $\tau=\xi\circ v: E_*\rightarrow E_*$. This is a continuous map (see Proposition \ref{prop7}). Invoking the Schauder fixed point theorem, we can find $\hat{u}\in E_*$ such that
	$$\hat{u}=\tau(\hat{u}).$$
	
	From the definition of $H(\cdot)$ we see that $\hat{u}\in W_p(T)$ is a solution of the nonconvex problem. The proof is now complete.
\end{proof}

Since we have the result for the strongly monotone case, we can now pass to the general setting.
\begin{theorem}\label{th12}
	If hypotheses $H(A)_1,H(\varphi)_1$ and $H(F)_2,H_0$ or $H(F)'_2$ hold, then problem (\ref{eq1}) admits a solution $\hat{u}\in W_p(T)$.
\end{theorem}
\begin{proof}
	Let $\epsilon_n\rightarrow 0^+$ and consider the approximate problem
	$$\left\{\begin{array}{l}
		-u'(t)\in A(t,u(t))+\epsilon_n\mathcal{F}(u(t))+\partial\varphi(u(t))+F(t,u(t))\ \mbox{for almost all}\ t\in T,\\
		u(0)=u(b).
	\end{array}\right\}$$
	
	This problem satisfies the conditions of Theorem \ref{th11} and we obtain a solution $u_n\in W_p(T)$ for all $n\in\NN$. From the proof of Theorem \ref{th11}, we know that $\{u_n\}_{n\geq 1}\subseteq E_*$. Hence $\{u_n\}_{n\geq 1}\subseteq W_p(T)$ is bounded and given the compactness of $E_*\subseteq C(T,H)$, we may assume that
	$$u_n\stackrel{w}{\rightarrow}\hat{u}\ \mbox{in}\ W_p(T)\ \mbox{and}\ u_n\rightarrow\hat{u}\ \mbox{in}\ C(T,H).$$
	
	We have that
	$$u_n=(\xi_n\circ v)(u_n)\ \mbox{for all}\ n\in\NN$$
	with $\xi_n(\cdot)$ being the solution map for the approximate problem (see Proposition \ref{prop7}) and $v(\cdot)$ is the continuous selection of the multifunction $H(\cdot)$ (see the proof of Theorem \ref{th11}). We have
	$$\left\{\begin{array}{l}
		-u'_n(t)=A(t,u_n(t))+\epsilon_n\mathcal{F}(u_n(t))+g_n(t)+v(u_n)(t)\ \mbox{for almost all}\ t\in T,\\
		u_n(0)=u_n(b)
	\end{array}\right\}$$
	with $g_n\in\partial\Phi(u_n),n\in\NN$. Since $\partial\phi(\cdot)$ maps bounded sets to bounded sets we may assume that
	$$g_n\stackrel{w}{\rightarrow}g\ \mbox{in}\ L^{p'}(T,H)\ \mbox{and}\ g\in\partial\Phi(u)$$
	(since $\partial\Phi(\cdot)$ is maximal monotone). As in the proof of Proposition \ref{prop3}, using the $L$-pseudo\-mo\-no\-tonicity of $a(\cdot)$, in the limit as $n\rightarrow\infty$, we obtain
	$$\left\{\begin{array}{l}
		-u'(t)=A(t,u(t))+g(t)+v(u)(t)\ \mbox{almost everywhere on}\ T,\\
		u(0)=u(b).
	\end{array}\right\}$$
	We conclude that $u\in W_p(T)$ is a solution of the nonconvex problem.
\end{proof}

\section{Extremal trajectories}

In this section we establish the existence of extremal periodic trajectories for problem (\ref{eq1}), that is, solutions which move through the extreme points of the multivalued perturbation $F(t,x)$. We know that even if $F(t,\cdot)$ is regular, the multifunction $x\mapsto {\rm ext}\, F(t,x)$ assigning the extreme points of $F(t,x)$, need not have any continuity properties (see Hu \& Papageorgiou \cite[Section 2.4]{4}).

In this section, the problem under consideration is the following:
\begin{eqnarray}\label{eq38}
	\left\{\begin{array}{l}
		-u'(t)\in A(t,u(t))+\partial\varphi(u(t))+{\rm ext}\, F(t,u(t))\ \mbox{for almost all}\ t\in T,\\
		u(0)=u(b).
	\end{array}\right\}
\end{eqnarray}

To be able to solve (\ref{eq38}), we need to strengthen the conditions on the multifunction $F(t,x):$

\smallskip
$H(F)_3:$ $F:T\times H\rightarrow P_{fc}(H)$ is a multifunction such that
\begin{itemize}
	\item[(i)] for all $x\in H,\ t\mapsto F(t,x)$ is graph measurable;
	\item[(ii)] for all $t\in T, F(t,\cdot)$ is $h$-continuous;
	\item[(iii)] there exists $M>0$ such that
	\begin{eqnarray*}
		&&0\leq(h,x)\ \mbox{for almost all}\ t\in T,\ \mbox{and all}\ |x|=M,\  h\in F(t,x),\\
		&&|F(t,x)|\leq a_M(t)\ \mbox{for almost all}\ t\in T,\ \mbox{and all}\ |x|\leq M,\ \mbox{with}\ a_M\in L^{p'}(T).
	\end{eqnarray*}
\end{itemize}

Alternatively, we can assume the following:

\smallskip
$H(F)'_3:$ $F:T\times H\rightarrow P_{fc}(H)$ is a multifunction such that hypotheses $H(F)'_3(i),(ii)$ are the same as the corresponding hypotheses $H(F)_3(i),(ii)$ and
\begin{itemize}
	\item[(iii)] $|F(t,x)|\leq k(t)[1+|x|]$ for almost all $t\in T$, and all $x\in H$, with $k\in L^{p'}(T)$
\end{itemize}

Again, we first deal with problem (\ref{eq38}) under the strong monotonicity conditions on the data.
\begin{theorem}\label{th13}
	If hypotheses $H(A),\ H(\varphi)$ or $H(A)',\ H(\varphi)'$ and $H(F)_3,\ H_0$ or $H(F)'_3$ hold, then problem (\ref{eq38}) admits a solution $\hat{u}\in W_p(T)$.
\end{theorem}
\begin{proof}
	The {\it a priori} bounds from Proposition \ref{prop6}, allow us to replace $F(t,x)$ by $\hat{F}(t,x)=F(t,p_M(x))$. So, without any loss of generality, we may assume that
	\begin{equation}\label{eq39}
		|F(t,x)|\leq\hat{\eta}(t)\ \mbox{for almost all}\ t\in T,\ \mbox{and all}\ x\in H,\ \mbox{with}\ \hat{\eta}\in L^{p'}(T).
	\end{equation}
	
	As before, we introduce the set
	$$W=\{h\in L^{p'}(T,H):|h(t)|\leq\hat{\eta}(t)\ \mbox{for almost all}\ t\in T\}$$
	and we define
	$$E_*=\overline{\rm conv}\,\xi(W)\in P_{kc}(C(T,H)).$$
	
	Theorem 8.31 of Hu \& Papageorgiou \cite[p. 260]{4} implies that there exists a continuous map $\gamma:E_*\rightarrow L^1_w(T,H)$ such that
	\begin{eqnarray}\label{eq40}
		&&\gamma(u)\in {\rm ext}\, S^{p'}_{F(\cdot,u(\cdot))}=S^{p'}_{{\rm ext}\, F(\cdot,u(\cdot))}\ \mbox{for all}\ u\in E_*\\
		&&(\mbox{see Hu \& Papageorgiou \cite[Theorem 4.5, p. 191]{4}}).\nonumber
	\end{eqnarray}
	
	We consider the map  $\sigma=\xi\circ\gamma:E_*\rightarrow E_*$ (see (\ref{eq39})). Suppose that $u_n\rightarrow u$ in $E_*$. Then $\gamma(u_n)\rightarrow\gamma(u)$ in $L^1_w(T,H)$. Invoking Lemma 2.8 of Hu \& Papageorgiou \cite[p. 24]{5}, we have $\gamma(u_n)\stackrel{w}{\rightarrow}\gamma(u)$ in $L^{p'}(T,H)$. Then by Proposition \ref{prop7}, we have
	\begin{eqnarray*}
		&&\sigma(u_n)=\xi(\gamma(u_n))\rightarrow\xi(\gamma(u))=\sigma(u)\ \mbox{in}\ C(T,H),\\
		&\Rightarrow&\sigma: E_*\rightarrow E_*\ \mbox{is continuous}.
	\end{eqnarray*}
	
	Since $E_*\in P_{kc}(C(T,H))$, we can apply the Schauder fixed point theorem and find $\hat{u}\in E_*$ such that
	\begin{eqnarray*}
		&&\hat{u}=\sigma(\hat{u}),\\
		&\Rightarrow & \hat{u}\in W_p(T)\ \mbox{is a solution of problem (\ref{eq38}) (see (\ref{eq40}))}.
	\end{eqnarray*}
The proof of Theorem \ref{th13} is complete.
\end{proof}

Next, we remove the strong monotonicity condition.
\begin{theorem}\label{th14}
	If hypotheses $H(A)_1,\ H(\varphi)_1$ and $H(F)_3,\ H_0$ or $H(F)'_3$ hold, then problem (\ref{eq38}) admits a solution $\hat{u}\in W_p(T)$.
\end{theorem}
\begin{proof}
	Again we choose $\epsilon_n\rightarrow 0^+$ and consider the approximate problems
	$$\left\{\begin{array}{l}
		-u'(t)\in A(t,u(t))+\epsilon_n\mathcal{F}(u(t))+\partial\varphi(u(t))+{\rm ext}\, F(t,u(t))\ \mbox{for almost all}\ t\in T,\\
		u(0)=u(b).
	\end{array}\right\}$$
	
	This problem satisfies the strong monotonicity condition and so Theorem \ref{th13} can be applied to produce a solution $u_n\in W_p(T)$ for all $n\in\NN$. We have
	$$\{u_n\}_{n\geq 1}\subseteq E_*\ \mbox{and}\ u_n=(\xi\circ\gamma)(u_n)\ \mbox{for all}\ n\in\NN.$$
	
	Therefore $\{u_n\}_{n\geq 1}$ is bounded in $W_p(T)$ and relatively compact in $C(T,H)$. So, we may assume that
	$$u_n\stackrel{w}{\rightarrow}\hat{u}\ \mbox{in}\ W_p(T)\ \mbox{and}\ u_n\rightarrow\hat{u}\ \mbox{in}\ C(T,H)\ \mbox{as}\ n\rightarrow\infty.$$
	
	We have
	\begin{eqnarray*}
		&&\gamma(u_n)\rightarrow\gamma(\hat{u})\ \mbox{in}\ L^1_w(T,H),\\
		&\Rightarrow&\gamma(u_n)\stackrel{w}{\rightarrow}\gamma(\hat{u})\ \mbox{in}\ L^{p'}(T,H)\\
		&&(\mbox{see Hu \& Papageorgiou \cite[Lemma 2.8, p. 24]{5} and (\ref{eq39})}).
	\end{eqnarray*}
	
	We know that for every $n\in\NN$
	\begin{eqnarray*}
		&&-u'_n(t)\in A(t,u_n(t))+\epsilon_n\mathcal{F}(u_n(t))+g_n(t)+\gamma(u_n)(t)\ \mbox{for almost all}\ t\in T,\\
		&&u_n(0)=u_n(b)
	\end{eqnarray*}
	with $g_n\in\partial\Phi(u_n),\ n\in\NN$. Since $\partial\varphi(\cdot)$ maps bounded sets to bounded sets and it is maximal monotone, we have (at least for a subsequence)
	$$g_n\stackrel{w}{\rightarrow}g\ \mbox{in}\ L^{p'}(T,H)\ \mbox{and}\ g\in\partial\Phi(\hat{u}).$$
	
	Passing to the limit as $n\rightarrow\infty$ in the evolution equation and using the $L$-pseudo-monotonicity of $a(\cdot)$, as before, we obtain
	\begin{eqnarray*}
		&&-\hat{u}'(t)\in A(t,\hat{u}(t))+g(t)+\gamma(\hat{u})(t)\ \mbox{for almost all}\ t\in T,\\
		&&u(0)=u(b),\\
		&\Rightarrow&\hat{u}\in W_p(T)\ \mbox{is a solution of (\ref{eq38})}.
	\end{eqnarray*}
This completes the proof of Theorem \ref{th14}.
\end{proof}

\section{Strong relaxation}

In this section we show that every solution of the convex problem can be approximated in the $C(T,H)$-norm topology by certain extremal trajectories. Such a result is known as ``strong relaxation" and is important for many applications. For example, in control theory it is related to the so-called ``bang-bang principle". In this context the result says that any state of the control system can be approximated by states which are generated by bang-bang controls. So, in the operation of the system, we can economize in the use of control functions.

To prove such an approximation result, we need to strengthen the conditions on the multivalued perturbation $F(t,x)$. So, the hypotheses are the following:

\smallskip
$H(F)_4:$ $F:T\times X\rightarrow P_{wkc}(H)$ is a multifunction such that
\begin{itemize}
	\item[(i)] for all $x\in H,\ t\mapsto F(t,x)$ is graph measurable;
	\item[(ii)] $h(F(t,x),F(t,y))\leq l(t)|x-y|$ for almost all $t\in T$, and all $x,y\in H$, with $l\in L^1(T)$;
	\item[(iii)] there exists $M>0$ such that
	\begin{eqnarray*}
		&&0\leq(h,x)\ \mbox{for almost all}\ t\in T,\ \mbox{and all}\ |x|=M,\ h\in F(t,x),\\
		&&|F(t,x)|\leq a_M(t)\ \mbox{for almost all}\ t\in T,\ \mbox{and all}\ |x|\leq M,\ \mbox{with}\ a_M\in L^{p'}(T).
	\end{eqnarray*}
\end{itemize}

Alternatively, we can impose the following conditions on $F(t,x)$.

$H(F)'_1:$ $F:T\times H\rightarrow P_{wkc}(H)$ is a multifunction such that hypotheses $H(F)'_4(i),(ii)$ are the same as the corresponding hypotheses $H(F)_4(i),(ii)$ and
\begin{itemize}
	\item[(iii)] $|F(t,x)|\leq k(t)[1+|x|]$ for almost all $t\in T$, and all $x\in H$, with $k\in L^{p'}(T)$.
\end{itemize}

In what follows, we denote by $\hat{S}_c\subseteq W_p(T)$ the solution set of problem (\ref{eq1}) with the multivalued perturbation $F(t,x)$ being convex-valued. Suppose $u\in \hat{S}_c$. Then by $\hat{S}_e(u(0))$ we denote the solution set of the following Cauchy problem
$$-v'(t)\in A(t,v(t))+\partial\varphi(v(t))+{\rm ext}\, F(t,v(t))\ \mbox{for almost all}\ t\in T,\ v(0)=u(0).$$

Reasoning as in the proof of Theorem \ref{th13}, we show that $\hat{S}_e(u(0))\subseteq W_p(T)$ is nonempty.
\begin{theorem}\label{th15}
	If hypotheses $H(A)_1,\ H(\varphi)_1$ and $H(F)_4,\ H_0$ or $H(F)'_4$ hold and $u\in \hat{S}_c$, then we can find $\{u_n\}_{n\geq 1}\subseteq \hat{S}_e(u(0))$ such that $u_n\rightarrow u$ in $C(T,H)$.
\end{theorem}
\begin{proof}
	As before, as a result of the {\it a priori} bounds established in Proposition \ref{prop6}, we may assume without any loss of generality that
	\begin{equation}\label{eq41}
		|F(t,x)|\leq\hat{\eta}(t)\ \mbox{for almost all}\ t\in T,\ \mbox{and all}\ x\in H,\ \mbox{with}\ \hat{\eta}\in L^{p'}(T).
	\end{equation}
	
	Since $u\in\hat{S}_c$, there exists $h\in S^{p'}_{F(\cdot,u(\cdot))}$ such that
	$$\left\{\begin{array}{l}
		-u'(t)\in A(t,u(t))+\partial\varphi(u(t))+h(t)\ \mbox{for almost all}\ t\in T,\\
		u(0)=u(b).
	\end{array}\right\}$$
	
	Again, we introduce the following two sets
	\begin{eqnarray*}
		&&W=\{h\in L^{p'}(T,H):|h(t)|\leq\hat{\eta}(t)\ \mbox{for almost all}\ t\in T\},\\
		&&E_*=\overline{\rm conv}\,\xi(W)\in P_{kc}(C(T,H)).
	\end{eqnarray*}
	
	Given $v\in E_*$ and $\epsilon>0$, we consider the multifunction $\Gamma_{v,\epsilon}:T\rightarrow 2^H\backslash\{\emptyset\}$ defined by
	$$\Gamma_{v,\epsilon}(t)=\left\{y\in F(t,v(t)):|h(t)-y|<\frac{\epsilon}{2Mb}+d(h(t),F(t,v(t)))\right\}.$$
	
	Here, $M>0$ is the {\it a priori} bound from Proposition \ref{prop6}. Hypotheses $H(F)_4(i),(ii)$ imply that the mapping $t\mapsto F(t,v(t))$ is measurable. It follows that $t\mapsto \Gamma_{v,\epsilon}(t)$ is graph measurable and by invoking the Yankov-von Neumann-Aumann selection theorem, we can find a measurable function $\hat{h}_{v,\epsilon}:T\rightarrow H$ such that
	\begin{eqnarray}\label{eq42}
		&&\hat{h}_{v,\epsilon}(t)\in\Gamma_{v,\epsilon}(t)\ \mbox{for almost all}\ t\in T,\nonumber\\
		&\Rightarrow&\hat{h}_{v,\epsilon}\in L^{p'}(T,H)\ (\mbox{see (\ref{eq41})}).
	\end{eqnarray}
	
	Therefore, if we introduce the multifunction $\hat{H}_{\epsilon}:E_*\rightarrow 2^{L^{p'}(T,H)}$ defined by
	$$\hat{H}_{\epsilon}(v)=S^{p'}_{\Gamma_{v,\epsilon}},$$
	then from (\ref{eq42}) we see that $\hat{H}_{\epsilon}(v)\neq\emptyset$ for all $v\in E_*$. In addition, Lemma 8.3 of Hu \& Papageorgiou \cite[p. 239]{4} implies that
	\begin{eqnarray*}
		&&v\rightarrow\hat{H}_{\epsilon}(v)\ \mbox{is lsc},\\
		&\Rightarrow&v\rightarrow\overline{\hat{H}_{\epsilon}}^{|\cdot|}(v)\ \mbox{is lsc}\\
		&&(\mbox{see Hu \& Papageorgiou \cite[Proposition 2.38, p. 50]{4}}).
	\end{eqnarray*}
	
	Moreover, $v\mapsto \overline{\hat{H}_{\epsilon}}^{|\cdot|}(v)$ has decomposable values. So, using the Bressan-Colombo \cite{1} selection theorem, we produce a continuous map $\gamma_{\epsilon}:E_*\rightarrow L^{p'}(T,H)$ such that
	$$\gamma_{\epsilon}(v)\in \overline{\hat{H}_{\epsilon}}^{|\cdot|}(v)\ \mbox{for all}\ v\in E_*.$$
	
	Then Theorem 8.31 of Hu \& Papageorgiou \cite[p. 260]{4} gives a continuous map $\beta_{\epsilon}:E_*\rightarrow L^1_w(T,H)$ such that
	\begin{equation}\label{eq43}
		\beta_{\epsilon}(v)\in {\rm ext}\, S^{p'}_{F(\cdot,v(\cdot))}=S^{p'}_{{\rm ext}\, F(\cdot,v(\cdot))}\ \mbox{and}\ ||\beta_{\epsilon}(v)-\gamma_{\epsilon}(v)||_w<\epsilon\ \mbox{for all}\ v\in E_*.
	\end{equation}
	
	Now let $\epsilon_n=\frac{1}{n},\ \gamma_n=\gamma_{\epsilon_n},\ \beta_n=\beta_{\epsilon_n}$ for all $n\in\NN$ and $u_0=u(0)=u(b)$. We cosider the following Cauchy problem
	\begin{equation}\label{eq44}
		-u'(t)\in A(t,u(t))+\partial\varphi(u(t))+\beta_n(u(t))\ \mbox{for almost all}\ t\in T,\ u(0)=u_0.
	\end{equation}
	
	Let $u_n\in W_p(T)$ be a solution of (\ref{eq43}). It is clear that $\{u_n\}_{n\geq 1}\subseteq\hat{S}_e(u_0)$. Therefore $\{u_n\}_{n\geq 1}\subseteq C(T,H)$ is relatively compact. Also, directly from (\ref{eq44}) we see that $\{u_n\}_{n\geq 1}\subseteq W_p(T)$ is bounded. So, we may assume that
	\begin{equation}\label{eq45}
		u_n\stackrel{w}{\rightarrow}\hat{u}\ \mbox{in}\ W_p(T)\ \mbox{and}\ u_n\rightarrow\hat{u}\ \mbox{in}\ C(T,H)\ \mbox{as}\ n\rightarrow\infty .
	\end{equation}
	
	From (\ref{eq32}) (with $\epsilon_m=0$) and since $u_n(0)=u_0=u(0)$ for all $n\in\NN$, we have
	\begin{eqnarray}\label{eq46}
		|u_n(t)-u(t)|^2&\leq&\int^t_0(\beta_n(u_n)(s)-h(s),u_n(s)-u(s))ds\nonumber\\
		&\leq& \int^t_0(\beta_n(u_n)(s)-\gamma_n(u_n)(s),u_n(s)-u(s))ds\nonumber\\
		&+&\int^t_0|\gamma_n(u_n)(s)-h(s)|\cdot|u_n(s)-u(s)|ds\nonumber\\
		&\leq&\int^t_0(\beta_n(u_n)(s)-\gamma_n(u_n)(s),u_n(s)-u(s))ds\nonumber\\
		&+&\int^t_0\left[\frac{1}{2Mbn}+d(h(s),F(s,u_n(s)))\right]|u_n(s)-u(s)|ds\nonumber\\
		&\leq&\int^t_0(\beta_n(u_n)(s)-\gamma_n(u_n)(s),u_n(s)-u(s))ds\nonumber\\
		&+&\frac{1}{n}+\int^t_0h(F(s,u(s)),F(s,u_n(s)))|u_n(s)-u(s)|ds\nonumber\\
		&\leq&\int^t_0(\beta_n(u_n)(s)-\gamma_n(u_n)(s),u_n(s)-u(s))ds+\epsilon+\nonumber\\
		&&\hspace{3cm}\int^t_0 l(s)|u_n(s)-u(s)|^2ds.
	\end{eqnarray}
	
	By (\ref{eq43}) and Lemma 2.8 of Hu \& Papageorgiou \cite[p. 24]{5}, we have
	\begin{eqnarray}\label{eq47}
		&&\beta_n(u_n)-\gamma_n(u_n)\stackrel{w}{\rightarrow}0\ \mbox{in}\ L^{p'}(T,H)\ \mbox{as}\ n\rightarrow\infty,\nonumber\\
		&\Rightarrow&\int^b_0(\beta_n(u_n)(s)-\gamma_n(u_n)(s),u_n(s)-u(s))ds\rightarrow 0\ \mbox{as}\ n\rightarrow\infty\ (\mbox{see (\ref{eq45})}).
	\end{eqnarray}
	
	Therefore, if in (\ref{eq46}) we pass to the limit as $n\rightarrow\infty$ and use (\ref{eq45}) and (\ref{eq47}), then
	\begin{eqnarray*}
		&&|\hat{u}(t)-u(t)|^2\leq \int^t_0l(s)|\hat{u}(s)-u(s)|^2ds,\\
		&\Rightarrow& \hat{u}=u\ (\mbox{by Gronwall's inequality}).
	\end{eqnarray*}

	Hence $u=\lim_{n\rightarrow\infty} u_n$ in $C(T,H)$ with $u_n\in \hat{S}_{e}(u(0))$ for all $n\in\NN$.
\end{proof}

\section{Examples}

In this section we illustrate the previous results, by considering parabolic distributed parameter control systems.

Let $T=[0,b]$ and assume that $\Omega\subseteq \RR^N$ is a bounded domain with Lipschitz boundary $\partial\Omega$. We consider the following nonlinear control system
\begin{eqnarray}\label{eq48}
	\left\{\begin{array}{l}
		\frac{\partial u}{\partial t}-\Delta_pu+\beta(u)\ni f_0(t,z,u)+(k(t,z),v(t,z))_{\RR^N}\ \mbox{in}\ (0,b)\times\Omega,\\
		u(t,\cdot)|_{\partial\Omega}=0\ \mbox{for all}\ t\in(0,b),\ u(0,\cdot)=u(b,\cdot)\ \mbox{in}\ \Omega,\\
		v(t,z)\in K(t,z)\ \mbox{for almost all}\ (t,z)\in T\times\Omega.
	\end{array}\right\}
\end{eqnarray}

In this problem, $\Delta_p\ (2\leq p<\infty)$ denotes the $p$-Laplacian differential operator defined by
$$\Delta_pu={\rm div}\,(|Du|^{p-2}Du)\ \mbox{for all}\ u\in W^{1,p}(\Omega).$$

The nonlinearity $f_0:T\times\Omega\times\RR\rightarrow\RR$ is a Carath\'eodory function, that is,
\begin{itemize}
	\item for all $x\in\RR$, the mapping $(t,z)\mapsto f_0(t,z,x)$ is measurable;
	\item for almost all $(t,z)\in T\times\Omega$, $x\mapsto f_0(t,z,x)$ is continuous;
	\item $|f_0(t,z,x)|\leq k_0(t,z)(1+|x|)$ for almost all $(t,z)\in T\times\Omega$, and all $x\in\RR$, with $k_0\in L^2(T\times \Omega)$.
\end{itemize}

Also, $\beta:\RR\rightarrow 2^{\RR}$ is a maximal monotone map. Then
$$\beta=\partial j\ \mbox{with}\ j\in\Gamma_0(\RR)$$
(see Corollary 3.2.40 of Papageorgiou \& Kyritsi Yiallourou \cite[p. 176]{7}). We set
$$\varphi(u)=\left\{\begin{array}{ll}
	\int_{\Omega}j(u(z))dz&\mbox{if}\ j(u(\cdot))\in L^1(\Omega)\\
	+\infty&\mbox{otherwise}
\end{array}\right.\ \mbox{for all}\ u\in L^2(\Omega).$$

We know that $\varphi\in\Gamma(L^2(\Omega))$ and
$$y\in\partial\varphi(u)\ \mbox{if and only if}\ y(z)\in\partial j(u(z))=\beta(u(z))\ \mbox{for almost all}\ z\in\Omega$$
(see Hu \& Papageorgiou \cite{4}). We assume that if $B\subseteq L^2(\Omega)$ is bounded, then $\int_\Omega j(u(z))dz\leq M$ for all $u\in B$, some $M>0$,  and for every $u\in L^p(T,W^{1,p}_{0}(\Omega))$, the multifunction $(t,z)\mapsto\beta(u(t,z))$ has a selection in the space $L^{p'}(T,L^2(\Omega))$. Since $p\geq 2$, this is satisfied if $(t,z)\mapsto\beta(u(t,z))$ admits a selection in $L^2(T\times\Omega)$.

The function $v\in L^2(T\times\Omega),\ v:T\times\Omega\rightarrow\RR^m$ is the control function and $K(t,z)\subseteq\RR^m$ is the control constant set. We assume that the multifunction $K:T\times\Omega\rightarrow P_{kc}(\RR^m)$ is graph measurable and $|K(t,z)|\leq\hat{M}$ for some $\hat{M}>0$ and for almost all $(t,z)\in T\times\Omega$.

We formulate problem (\ref{eq48}) in the form of an abstract evolution inclusion as (\ref{eq1}). The evolution triple consists of the following spaces
$$X=W^{1,p}_{0}(\Omega),\ H=L^2(\Omega),\ X^*=W^{-1,p'}(\Omega)\ \left(\frac{1}{p}+\frac{1}{p'}=1\right).$$

The Sobolev embedding theorem implies that $X\hookrightarrow H$ compactly. Let $A:X\rightarrow X^*$ be the nonlinear map defined by
$$\left\langle A(u),h\right\rangle=\int_{\Omega}|Du|^{p-2}(Du,Dh)_{\RR^N}dz.$$

Evidently, $A(\cdot)$ is continuous, strictly monotone, hence maximal monotone, too. Also, we have
\begin{eqnarray*}
	&&\left\langle A(u),u\right\rangle=||Du||^p_p=||u||^p\ (\mbox{by the Poincar\'e inequality})\\
	\mbox{and}&&\left\langle A(u),h\right\rangle\leq||u||^{p-1}||h||_p\ \mbox{for all}\ h\in W^{1,p}_{0}(\Omega)\ (\mbox{by H\"older's inequality}).
\end{eqnarray*}

So, hypotheses $H(A)_1$ are satisfied.

Let $f:T\times H\rightarrow H$ be the Nemitsky map corresponding to the function $f_0$, that is,
$$f(t,u)(\cdot)=f_0(t,\cdot,u(\cdot))\ \mbox{for all}\ u\in H=L^2(\Omega).$$

We introduce the multifunction $G:T\rightarrow P_{wkc}(L^2(\Omega))$ defined by
$$G(t)=\{(k(t,\cdot),v(\cdot))_{\RR^m}:v\in S^2_{K(t,\cdot)}\}.$$

Then $G(\cdot)$ is measurable and $|G(\cdot)|\in L^2(T)$. We set
$$F(t,u)=f(t,u)+G(t)\ \mbox{for all}\ (t,u)\in T\times L^2(\Omega).$$

It follows that this multifunction satisfies hypotheses $H(F)'_1$. Then problem (\ref{eq48}) is equivalent to the following nonlinear evolution inclusion
$$\left\{\begin{array}{l}
	-u'(t)\in A(u(t))+\partial\varphi(u(t))+F(t,u(t))\ \mbox{for almost all}\ t\in T,\\
	u(0)=u(b).
\end{array}\right\}$$

By Theorem \ref{th9}, this problem has a solution $u\in L^p(T,W^{1,p}_{0}(\Omega))$ such that
$$\frac{\partial u}{\partial t}\in L ^{p'}(T,W^{-1,p'}(\Omega)).$$

In fact, the set of solutions is compact in $C(T,L^2(\Omega))$ (see Theorem \ref{th10}). Moreover, if we assume that
$$|f_0(t,z,x)-f_0(t,z,y)|\leq l_0(t,z)|x-y|\ \mbox{for almost all}\ (t,z)\in T\times\Omega,\ \mbox{and all}\ x,y\in\RR,$$
then by the strong relaxation theorem (see Theorem \ref{th15}), given any solution $u$ of the convex problem, we can find a solution $\hat{u}$ which is generated by a bang-bang control $v(t,z)\in {\rm ext}\, V(t,z)$ for almost all $(t,z)\in T\times\Omega$ such that
$$\sup\limits_{t\in T}|u(t,\cdot)-\hat{u}(t,\cdot)|_{L^2(\Omega)}<\epsilon,\ \epsilon>0$$

In a similar way, we can also deal with the following control system
$$\left\{\begin{array}{l}
	\frac{\partial u}{\partial t}-\Delta_pu-\Delta u=f_0(t,z,u)+(k(t,z),v(t,z))_{\RR^m}\ \mbox{in}\ (0,b)\times\Omega,\\
	0\in\beta(u(t,z))\ \mbox{for all}\ (t,z)\in T\times\partial\Omega,\\
	u(0,z)=u(b,z)\ \mbox{for almost all}\ z\in\Omega,v(t,z)\in K(t,z)\ \mbox{almost everywhere in}\ T\times\Omega.
\end{array}\right\}$$

Note that in this case hypothesis $H(A)$ is satisfied.

\medskip
{\bf Acknowledgments.} This research was supported by the Slovenian Research Agency grants P1-0292, J1-8131, J1-7025, and N1-0064. V.D. R\u adulescu acknowledges the support through a grant of the Romanian Ministry of Research and
Innovation, CNCS - UEFISCDI, project number PN-III-P4-ID-PCE-2016-0130, within
PNCDI III.


\begin{thebibliography}{99}

\bibitem{1} A. Bressan, G. Colombo, Extensions and selections of maps with decomposable values, {\it Studia Math.} {\bf 90} (1988), 69-85.

\bibitem{2} L. Egghe, {\it Stopping Time Techniques for Analysts and Probabilists}, Cambridge University Press, Cambridge, 2006.

\bibitem{3} L. Gasinski, N.S. Papageorgiou, {\it Nonlinear Analysis}, Chapman \& Hall/CRC, Boca Raton, Fl, 2006.

\bibitem{4} S. Hu, N.S. Papageorgiou, {\it Handbook of Multivalued Analysis. Volume I: Theory}, Kluwer Academic Publishers, Dordrecht,  1997.

\bibitem{5} S. Hu, N.S. Papageorgiou, {\it Handbook of Multivalued Analysis. Volume II: Applications}, Kluwer Academic Publishers, Dordrecht,  2000.

\bibitem{6} J.-L. Lions, {\it Quelques M\'ethodes de R\'esolution des Probl\`emes aux Limites Non Lin\'eaires}, Dunod, Paris, 1969.

\bibitem{7} N.S. Papageorgiou, S. Kyristsi Yiallourou, {\it Handbook of Applied Analysis}, Springer, New York, 2009.

\bibitem{8} N.S. Papageorgiou, F. Papalini, F. Renzacci,  Existence of solutions and periodic solutions for nonlinear evolution inclusions, {\it Rend. Circolo Mat. Palermo} {\bf 48} (1999), 341-364.

\bibitem{9} N.S. Papageorgiou, V.D. R\u adulescu, Periodic solutions for time-dependent subdifferential evolution inclusions, {\it Evolution Equations Control Theory} {\bf 6} (2017), 277-297.

\bibitem{9bis} N.S. Papageorgiou, V.D. R\u adulescu, D.D. Repov\v{s}, Sensitivity analysis for optimal control problems governed by nonlinear evolution inclusions, {\it Advances in Nonlinear Analysis} {\bf 6} (2017), 199-235.

\bibitem{10} R.T. Rockafellar, R. Wets, {\it Variational Analysis}, Springer-Verlag, Berlin, 1998.

\bibitem{roubicek} T. Roubicek, {\it Nonlinear Partial Differential Equations with Applications}, Birkh\"auser, Basel, 2005.

\bibitem{11} X. Xue, Y. Cheng, Existence of periodic solutions of nonlinear evolution inclusions in Banach spaces, {\it Nonlinear Analysis: Real World Appl.} {\bf 11} (2010), 459-471.

\bibitem{12} E. Zeidler, {\it Nonlinear Functional Analysis and its Applications II}, Springer, New York, 1990.


\end{thebibliography}
\end{document}